\documentclass[10pt]{amsart}
\usepackage{amsmath,graphicx,amssymb,amsthm}
\usepackage{color}
\def\A{\mathcal{A}}
\def\B{\mathcal B}

\def\B{\mathcal B}

\def\amslatex{$\mathcal{A}\kern-.1667em\lower.5ex\hbox{$\mathcal{M}$}\kern-.125em\mathcal{S}$-\LaTeX}

\newtheorem{set}{set}[section]

\newtheorem{Definition}[set]{Definition}

\newtheorem{Lemma}[set]{Lemma}

\newtheorem{Remark}[set]{Remark}
\newtheorem{Theorem}[set]{Theorem}
\newcommand{\define}{\mathrel{\hbox{$\equiv$\hskip -.90em \lower .47ex \hbox{$\leftharpoondown$}}}}
\newcommand{\enifed}{\mathrel{\hbox{$\equiv$\hskip -.90em \lower .47ex \hbox{$\rightharpoondown$}}}}

\begin{document}
\title{   Multidimensional Free Poisson Limits on Free Stochastic Integral Algebras}

\author{Mingchu Gao}
\address{School of Mathematics and Information Science,
Baoji University of Arts and Sciences,
Baoji, Shaanxi 720113,
China; and
Department of Mathematics,
Louisiana College, Pineville, LA 71359, USA} \email[Mingchu
Gao]{mingchu.gao@lacollege.edu}
\author{Junsheng Fang}
\address{School of Mathematics and Information Science,
Hebei Normal University,
Shijiazhuang, Hebei, 050024, China}
 \email[Junsheng Fang
]{junshengfang@hotmail.com}
\thanks{${}^*$ The second author was supported
 by the Project sponsored by the NSFC grant 11431011 and startup funding from Hebei Normal University}

\begin{abstract}
In this paper, we prove four-moment theorems for  multidimensional free Poisson limits on free Wigner chaos or the free Poisson algebra. We prove that, under mild technical conditions, a bi-indexed sequence of free stochastic integrals in free Wigner algebra or free Poisson algebra converges  to a free sequence of free Poisson random variables  if and only if  the  moments with order not greater than four of the sequence converge to the corresponding moments of the limit sequence of random variables.  Similar four-moment theorems hold when the limit sequence is not free, but has a multidimensional free Poisson distribution with parameters $\lambda>0$ and $\alpha=\{\alpha_i: 0\ne \alpha_i\in \mathbb{R}, i=1, 2, \cdots\}$.
\end{abstract}

\maketitle


\section*{Introduction}
 A basic question in free probability is to find distributions of random variables in a noncommutative probability space. No one can find  distributions of random variables in an arbitrary noncommutative probability space, but mathematicians want to investigate distributions of as many as possible random variables. The initial approach to this question is to study special  distributions inspired by the work in classical probability, such as free Gaussian distributions (i.e., semicircle distributions), free Poisson distributions, etc. But not too many distributions in classical probability could be defined and studied in free probability. Mathematicians extended their well-studied research territory by imposing operations such as free addition and free multiplication among the well known distributions. A more complicated and advanced operation is integration: integrating a deterministic (multi-variable) function with respect to a random measure of well known random variables such as semicircle random variables or free Poisson random variables. The resulting random variable is called the {\sl free stochastic (multiple) integral} of the function with respect to the random measure.
  The study on the distributions of free stochastic integrals is one of the main topics in free probability. Unfortunately, after integrating a function $f\in L^2(\mathbb{R}_+^q)$ ($q\ge 4$) with respect to a random measure of semicircular or free Poisson random variables, one can hardly get any non-trivial semicircle or free Poisson random variables (see Corollary 4.5 in \cite{BP}, Proposition 1.5 in \cite{NP}, Corollary 1.7 in \cite{KNPS}, and Corollary 1.6 in \cite{SB2}). A natural question, then, is to study the convergence of a sequence of free stochastic integrals. On the other hand, the limit theory (e.g., central and Poisson limit theorems) is a key topic in both classical and free probability. Mathematicians, thus,  began to study the convergence problem among (free) stochastic integrals more than a decade ago.

Let $\{W_t: t\ge 0\}$ be a standard Brownian motion on $\mathbb{R}_+$, and let $n\in \mathbb{N}$. Denote by $I^W(f)$ the multiple stochastic Wiener-It$\hat{o}$ integral of order $n$ of a function $f\in L^2(\mathbb{R}_+^n)$. Denoted by $L^2_s(\mathbb{R}_+^n)$ the subset of $L^2(\mathbb{R}_+^n)$ composed of symmetric functions. The collection of random variables $\{I^W(f): f\in L^2_s(\mathbb{R}_+^n)\}$ is called the $n$th Wiener chaos associated with $W$. In their seminal paper \cite{NuP}, Nualart and Peccati proved that, given  a sequence of elements with variance one living in a Wiener chaos of  a fixed order,  the convergence of this sequence to the standard normal distribution  is equivalent to the convergence of the fourth moment of the sequence to three. This result is now known as the {\sl fourth moment theorem}. The above result has led to a wide collection of new results and inspired several new research directions (see \cite{NP1} and the constantly updated web-page: http://sites.google.com/site/malliavinstein/home).     The fourth moment theorem was proved to hold as well for sequences of vectors of multiple integrals possibly of different orders in \cite{PT}.

Non-commutative counterparts of the results in classical probability have been established in the context of the chaos associated with a free Brownian motion and a free Poisson random measure. The mentioned concepts and notations in this section will be defined in Section 1.

{\bf Convergence of multiple integrals with respect to a free Brownian motion}. Let $I^S(f)$ be the multiple free stochastic integral of a function $f\in L^2(\mathbb{R}_+^q)$ ($q\ge 2$) with respect to a free Brownian motion $S$.  Kemp, Nourdin, Peccati, and Speicher \cite{KNPS} proved a beautiful fourth moment theorem for the convergence of a sequence of multiple Wigner integrals in law to the standard semicircular distribution (Theorem 1.3 in \cite{KNPS}). A similar result was obtained in \cite{NP} for such a sequence to converge to a centered free Poisson distribution (Theorem 1.4 in \cite{NP}). Nourdin, Peccati, and Speicher \cite{NPS} proved a fourth moment theorem for a sequence of multidimensional  free stochastic integral vectors to converge to a multidimensional semicircular limit theorem on the free Wigner chaos (Theorem 1.3 in \cite{NPS}).

{\bf Convergence of multiple integrals with respect to a free Poisson random measure}. Let $\widehat{N}$ be a centered free  Poisson random measure  on $\mathbb{R}_+$, and $I^{\widehat{N}}(f)$ be the multiple integral of $f\in L^2(\mathbb{R}_+^q)$ ($q\ge 2$). Bourguin and Peccati \cite{BP} proved a fourth moment theorem for convergence to a semicircular distribution on the free Poisson algebra (Theorem 4.3 in \cite{BP}).
In \cite{SB2}, S. Bourguin gave a fourth moment theorem for convergence to a free centered Poisson distribution  on the free Poisson algebra (Theorem 1.5 in \cite{SB2}).
He later proved a fourth moment theorem for convergence to vector-valued semicircular limits on the free Poisson algebra (Theorem 1.3 in \cite{SB1}).
Recently, Bourguin and Nourdin provided  an equivalent condition for the convergence discussed in \cite{SB1} in terms of the fourth moment of the Euclidean norm of the involved random vectors when the kernels of the multiple integrals are fully symmetric (Theorem 6.3 in \cite{BN}).

{\bf The work in the present paper}. Given a bi-indexed sequence of multiple free stochastic integrals with respect to free Wigner chaos or a free Poisson random measure, the aim of this paper is to investigate the convergence of the sequence to a multidimensional free Poisson distribution defined in \cite{AG}.  The rest of this paper is organized as follows. In Section 1, we recall relevant elements in free probability, free Wigner chaos, and free Poisson algebra used in sequel.   In Section 2, we study the convergence problem for a bi-indexed sequence of free stochastic integrals of  symmetric functions in $L^2(\mathbb{R}_+^q)$ ($q$ is even) with respect to free Bromnian motion. We prove that such a sequence  converges in distribution to a free sequence of free Poisson random variables if and only if its joint moments with order less than or equal to four converge to the corresponding moments of the limit sequence (Theorem 2.7). In section 3, a similar theorem is proved  when the free stochastic integrals are defined with respect to a centered free Poisson random measure and the kernels are symmetric, bounded, with bounded supports   (Theorem 3.3). Finally, in Section 4, we study the convergence problem of sequences studied in Sections 2 and 3, when the limit sequence has a multidimensional free Poisson distribution with parameters $\lambda>0$ and $\alpha=\{\alpha_1, \alpha_2, \cdots\}$, a sequence of non-zero real numbers. Using the techniques in Sections 2, 3, \cite{NP}, and \cite{SB2}, we get  four-moment theorems in this case (Theorems 4.2 and 4.4).

\section{Preliminaries}

In this section, we recall the  relevant elements in free probability (especially, free stochastic integration) used in sequel. For details on the subject, the reader is referred to \cite{SB2}, \cite{BP}, \cite{NS}, \cite{NP}, and \cite{VDN}. Basics on operator algebras can be found in \cite{KR}.



{\bf Symmetric functions and contractions.} For an integer $q\ge 2$, let $L^2(\mathbb{R}_+^q)$ be the space of all complex-valued functions on $\mathbb{R}_+^q=\{(x_1, \cdots, x_q): x_i\ge 0, i=1, 2, \cdots, q\}$ that are square-integrable with respect to $\mu$, the Lebesgus measure on $\mathbb{R}^q$. Given $f\in L^2(\mathbb{R}_+^q)$, we write $$f^*(t_1, t_2, \cdots, t_q)=\overline{f(t_q, \cdots, t_2, t_1)}.$$ We say that $f\in L^2(\mathbb{R}_+^q)$ is {\sl symmetric} if $f(t_1, \cdots, t_q)=f^*(t_1, \cdots, t_q)$, for all $(t_1, \cdots, t_q)\in \mathbb{R}_+^q$.


Let $m,n$ be natural numbers, $f\in L^2(\mathbb{R}_+^m), g\in L^2(\mathbb{R}_+^n)$, and $p\le \min\{m, n\}$ be a nonnegative integer. The {\sl $p$-th arc contraction} $f\stackrel{p}\smallfrown g$ is defined by
\begin{align*}
&f\stackrel{p}\smallfrown g(t_1, \cdots, t_{m+n-2p})\\
=&\int_{\mathbb{R}_+^p}f(t_1, \cdots, t_{m-p}, s_p, \cdots, s_1)g(s_1, \cdots, s_p, t_{m-p+1}, \cdots, t_{m+n-2p})ds_1ds_2\cdots ds_p.
\end{align*}
When $p=0$, we define $f\stackrel{0}\smallfrown g=f\otimes g$, the $m+n$-variable function defined by $$f\otimes g(t_1, \cdots, t_{m+n})=f(t_1, \cdots, t_m)g(t_{m+1}, \cdots, t_{m+n}).$$ When $p=m=n$, $f \stackrel{p}\smallfrown g=\langle g,f^*\rangle$.

The {\sl star contraction} of index $(p, p-1)$ of $f$ and $g$ is defined by
\begin{align*}
&f\star_{p}^{p-1} g(t_1, \cdots, t_{m+n-2p+1})\\
=&\int_{\mathbb{R}_+^{p-1}}f(t_1, \cdots, t_{m-p}, t_{m-p+1}, s_{p-1}, \cdots, s_1)\\
\times& g(s_1, \cdots, s_{p-1}, t_{m-p+1}, \cdots, t_{m+n-2p+1})ds_1ds_2\cdots ds_{p-1}.
\end{align*}
One can verify the relations
$$(f\star_p^{p-1}g)^*=g^*\star_p^{p-1}f^*, (f\stackrel{p}\smallfrown g)^*=g^*\stackrel{p}\smallfrown f^*.\eqno (1.0)$$

{\bf Free Probability,  Free Brownian motion,  and free Wigner Chaos.} The main references for the contents of this subsection are \cite{BS} and \cite{NP}. Let $(\A, \varphi)$ be a tracial $W^*$-probability space, that is, $\A$ is a finite von Neumann algebra and $\varphi:\A\rightarrow \mathbb{C}$ is a faithful, normal, and tracial state on $\A$. Self-adjoint operators in $\A$ are referred to {\sl random variables}. The {\sl law} (or {\sl distribution}) of a random variable $X\in \A$ is the unique Borel measure on $\mathbb{R}$ having the same moments as $X$ (see Proposition 3.13 in \cite{NS}). For $1\le p\le \infty$, one writes $L^p(\A, \varphi)$ to indicate the $L^p$ space obtained by completing $\A$ with respect to the norm $\|X\|_p=\varphi(|X|^p)^{1/p}$, where $|X|=(X^*X)^{1/2}$,  for $p<\infty$, and $\|\cdot\|_\infty$ stands for the operator norm.

Let $\A_1, \cdots, \A_n$ be unital subalgebras of $\A$. Let $X_1, \cdots, X_m$ be elements chosen among the $\A_i$'s such that, for $1\le j<m$, $X_j$ and $X_{j+1}$ do not come from the same $\A_i$, and such that $\varphi(X_j)=0$ for each $j$. The subalgebras $\A_1, \cdots, \A_n$ are said to be {\sl free} or {\sl freely independent} if, in this circumstance, $\varphi(X_1\cdots X_m)=0$. Random variables are {\sl free} if the unital algebras they generate are free.

 A {\sl partition} of $[n]=\{1, 2, \cdots, n\}$ is a collection of mutually disjoint nonempty subsets $B_1, B_2, \cdots$, $B_r$ of $[n]$ such that $B_1\cup B_2\cup \cdots \cup B_r=[n]$. The subsets are called the {\sl blocks} of the partition. The blocks are ordered by their least elements, i.e. $\min B_i<\min B_j$ iff $i<j$. The cardinality of $B_i$ is denoted by $|B_i|$. The set of all partitions of $[n]$ is denoted by $\mathcal{P}(n)$. 
A partition $\pi\in \mathcal{P}(n)$ has a {\sl crossing} if there are two distinct blocks $B_i$ and $B_j$ in $\pi$ with elements $x_i, y_i\in B$ and $x_j, y_j\in B_j$ such that $x_i<x_j<y_i<y_j$. If a partition $\pi$ has no crossings, we say that $\pi$ is a {\sl non-crossing partition}. The set of all  non-crossing partitions of $[n]$ is denoted by $NC(n)$. It is well known (see Page 144 in \cite{NS}) that the reversed refined order induces a lattice structure on $NC(n)$, where the partial order is defined by $\pi\le \sigma$, for $\pi, \sigma\in NC(n)$ if every block in $\sigma$ is the union of some blocks in $\pi$.

A set $\{a_i:i\in I\}$ of random variables is a free family of random variables  (or, we say, $a_i, i\in I$, are free) if and only if  for $1<n\in \mathbb{N}$ and $\chi:\{1, 2, \cdots, n\}\rightarrow I$, we have $\kappa_n(a_{\chi(1)}, a_{\chi(2)}, \cdots, a_{\chi(n)})=0$ whenever $\chi$ is not  constant (Theorem 11.20 in \cite{NS}).

The (centered) {\sl semicircular distribution} (or Wigner law) $S(0,t)$ is the probability distribution
$$S(0,t)(dx)=\frac{1}{2\pi t}\sqrt{4t-x^2}dx, |x|\le 2\sqrt{t}.$$ The odd moments of this distribution are $0$ and even moments are given by $$\int_{-2\sqrt{t}}^{2\sqrt{t}}x^{2m}S(0,t)(dx)=C_mt^m,$$ where $C_m=\frac{1}{m+1}\left(\begin{matrix}2m\\m\end{matrix}\right)$ is the $m$-th Catalan number (see Lecture 2 in \cite{NS}).

Let $d\ge 2$ be a natural number, and let  $c=(c(i,j))_{d\times d}$ be a positive definite symmetric matrix. A $d$-dimensional vector $(s_1, \cdots, s_d)$ of random variables in $(\A, \varphi)$ is said to be  {\sl a semicircular family with covariance $c$} if for every $n\ge 2$ and $i_1, \cdots, i_n\in \{1, \cdots, d\}$ $$\varphi(s_{i_1}s_{i_2}\cdots s_{i_d})=\sum_{\pi\in NC_2(n)}\prod_{(a,b)\in \pi}c(i_a, i_b),$$ where $NC_2(n)=\{\pi\in NC(n): |V|=2, \forall V\in \pi\}$ (see Definition 8.15 in \cite{NS}).

A {\sl free Brownian motion} $S$ consists of $(i)$ a filtration $\{\A_t:t\ge 0\}$, that is, a family of von Neumann subalgebras of $\A$ such that $\A_s\subset\A_t$ whenever $0\le s< t$, $(ii)$ a collection $S=\{S_t:t\ge 0\}$ of random variables in $\A$ such that: $(a)$ $S_0=0$ and $S_t\in \A_t$, for $t>0$, $(b)$ for every $t$, $S_t$ has the distribution $S(0, t)$, and $(c)$ for every $0<s<t$, the operator $S_t-S_s$ is free from $\A_s$, and has the distribution $S(0, t-s)$.

For a natural number $n$, and $f\in L^2(\mathbb{R}_+^n)$, the {\sl (free stochastic) multiple integral} $I^S(f)$ is defined as follows. $(a)$ Define $I^S(f)=(S_{b_1}-S_{a_1})\cdots (S_{b_n}-S_{a_n})$ for a function $$f(t_1, \cdots, t_n)=\mathbf{1}_{(a_1, b_1)}(t_1)\times \cdots \times \mathbf{1}_{(a_n, b_n)}(t_n),\eqno (1.1)$$ where intervals $(a_i, b_i), i=1, \cdots, n$, are pairwise disjoint; $(b)$ extend the above $I^S(f)$ linearly to the integrals of simple functions vanishing on diagonals, that is, linear combinations of functions of type $(1.1)$; $(c)$ exploit the isometric relation ((3.3) in \cite{NP}) $$\langle I^S(f), I^S(g)\rangle_{L^2(\A, \varphi)}=\langle f, g\rangle_{L^2(\mathbb{R}_+^n)},$$ where $f$ and $g$ are simple functions vanishing on diagonals, and use a density argument to define $I^S(f)\in L^2(\A, \varphi)$ for $f\in L^2(\mathbb{R}_+^n)$. We define $I^S(\alpha)=\alpha$, for a complex number $\alpha$. The set of all random variables $I^S(f), f\in L^2(\mathbb{R}_+^n)$, is called the $n$th {\sl free Wigner chaos} associated with $S$. For $m, n\in \mathbb{N}$, and $f\in L^2(\mathbb{R}_+^m), g\in L^2(\mathbb{R}_+^n)$, we have the following orthogonal and isometric formula (Proposition 3.5 in \cite{BP})
$$\varphi(I^S(f)I^S(g))=\delta_{m,n}\langle f, g^*\rangle, \eqno (1.2)$$ where $\delta_{m,n}=1$, if $m=n$; $\delta_{m,n}=0$, if $m\ne n$.

 {\sl Free Wigner chaos} is defined as the linear space of all multiple integrals $I^S(f)$, for $f\in L^2(\mathbb{R}_+^n)$, $n\ge 0$, where $L^2(\mathbb{R}_+^0)=\mathbb{C}$. Free Wigner chaos becomes a unital $*$-algebra by imposing the following operations
 $$I^S(f)I^S(g)=\sum_{k=0}^{\min\{m, n\}}I^S(f\stackrel{k}\smallfrown g),  I^S(f)^*=I^S(f^*), \forall f\in L^2(\mathbb{R}_+^m), g\in L^2(\mathbb{R}_+^n)\eqno (1.3)$$ (see Remark 3.4 in \cite{BP} and (3.4) in \cite{NP}). The multiplication unit of the free Wigner Chaos
 is $1\in L^2(\mathbb{R}^0)=\mathbb{C}$, since, $I^S(1)$ is defined as $1$.

{\bf The free Poisson algebra}. The main reference for this section is \cite{BP}.
A {\sl free Poisson distribution} $N(\lambda, \alpha)$ with parameters $\lambda>0$ and $\alpha\in \mathbb{R}$ with $\alpha\ne 0$ is a probability distribution on $\mathbb{R}$ defined as follows.  A random variable $X\in \A$ has a  distribution $N(\lambda, \alpha)$ if and only if $$\kappa_m(X)=\lambda\alpha^m, \forall m\ge 1,$$ where $\kappa_m$ is the $m$th free cumulant of $(\A, \varphi)$ (see Proposition 12.11 in \cite{NS}).
When $\alpha=1$, we denote $N(\lambda, \alpha)$ by $N(\lambda)$. If $X(\lambda)$ has a distribution $N(\lambda)$, we denote the distribution of $Y(\lambda):=X(\lambda)-\lambda\mathrm{1}$ by $\widehat{N}(\lambda)$, which is called the {\sl centered free Poisson distribution} of parameter $\lambda$, which is characterized by $$\kappa_1(Y(\lambda))=0, \kappa_m(Y(\lambda))=\lambda, m\ge 2.$$

Let $(Z,\mathcal{Z})$ be a Polish space with $\mathcal{Z}$ the associated Borel $\sigma$-field, and let $\mu$ be a positive $\sigma$-finite measure over $(Z, \mathcal{Z})$ without atoms. We denote by $\mathcal{Z}_\mu$ the class of those $A\in \mathcal{Z}$ such that $\mu(A)<\infty$. Let $(\A, \varphi)$ be a tracial $W^*$-probability space and $\A_+$ be the cone of positive operators in $\A$. Then, a {\sl free Poisson random measure} with control $\mu$ on $(Z, \mathcal{Z})$ and values in $(\A,\varphi)$ is a mapping $N:\mathcal{Z}_\mu\rightarrow \A_+$ with the following properties.
\begin{enumerate}
\item For $A\in \mathcal{Z}_\mu$, $N(A)$ is a free Poisson random variable with parameter $\mu(A)$.
\item If $2\le r\in \mathbb{N}$,  and $A_1, \cdots, A_r\in \mathcal{Z}_\mu$ are disjoint, then $N(A_1), \cdots,  N(A_r)$ are free, and $N(\cup_{i=1}^rA_i)=\sum_{i=1}^rN(A_i)$.
\end{enumerate}
 The existence of free Poisson measures on an appropriate $W^*$-probability space is guaranteed by Theorem 3.3 in \cite{BnT} and Theorem 5.1 in \cite{BV}. If $N$ is a free Poisson random measure with control $\mu$, we will denote by $\widehat{N}$ the mapping from $\mathcal{Z}_\mu$ into $\A$ defined by $\widehat{N}(A)=N(A)-\mu(A)\mathrm{1}, A\in \mathcal{Z}_\mu$, where $\mathrm{1}$ is the unit of $\A$.  Then   $\widehat{N}(A)$ has the centered free Poisson distribution $\widehat{N}(\mu(A))$. We call $\widehat{N}$ a {\sl free centered Poisson measure}. In this paper, we let $(Z, \mathcal{Z}, \mu)=(\mathbb{R}_+, \mathcal{B}(\mathbb{R}_+), \mu)$, where $\mu$ is the Lebesgue measure on $\mathbb{R}_+$. Then $\mathcal{Z}_\mu=\mathcal{B}_\mu(\mathbb{R}_+)$ is the set of all Borel subsets $B$ of $\mathbb{R}_+$ with finite Lebesgue measure ($\mu(B)<\infty$).

{\sl Centered free Charlier polynomials} are defined by the following recurrence relations
$$C_{0}(x,\lambda)=1, C_1(x,\lambda)=x, xC_m(x,\lambda)=C_{m+1}(x,\lambda)+C_m(x, \lambda)+\lambda C_{m-1}(x, \lambda), m\ge 1.$$

Let $q\ge 2$ be an integer, and let $\mathcal{E}_q$ be the linear subspace of $L^2(\mathbb{R}_+^q)$ generated by functions of the type $$f=\mathrm{1}_{A_1}^{\otimes k_1}\otimes \mathrm{1}_{A_2}^{\otimes k_2}\otimes\cdots\otimes \mathrm{1}_{A_l}^{\otimes k_l},\eqno (1.4)$$ where $k_1+\cdots+k_l=q$, each $A_j$ is bounded, and $A_j\cap A_{j+1}=\emptyset$, for $j=1, 2, \cdots, l-1$. For a function of the type (1.4), the {\sl multiple free stochastic integral} of $f$ with respect to the centered free Poisson random measure $\widehat{N}$ is defined by
$$I^{\widehat{N}}(f)=C_{k_1}(\widehat{N}(A_1), \mu(A_1))\cdots C_{k_l}(\widehat{N}(A_l), \mu(A_l)).$$ We extend $I^{\widehat{N}}(f)$ to general $f\in \mathcal{E}_q$ by linearity.

For $q, q'\ge 0$, and $f\in \mathcal{E}_q$ and $g\in \mathcal{E}_{q'}$, one has the following orthogonal and isometric property $$\varphi(I^{\widehat{N}}(g)^*I^{\widehat{N}}(f))=\langle f, g\rangle_{L^2(\mathbb{R}_+^q)}\delta_{q,q'}\eqno (1.5)$$ (see Proposition 3.5 in \cite{BP}). Then, for a $f\in L^2(\mathbb{R}_+^q)$, one can define $I^{\widehat{N}}(f)\in L^2(\A, \varphi)$  by the density of $\mathcal{E}_q$ in $L^2(\mathbb{R}_+^q)$ and the above isometry property. The space $L^2(\mathcal{X}(\widehat{N}), \varphi)=\{I^{\widehat{N}}(f): f\in L^2(\mathbb{R}_+^q), q\ge 0\}$ is a unital $*$-algebra with the following operations.
 For $f\in L^2(\mathbb{R}_+^p), g\in L^2(\mathbb{R}_+^q)$, $p, q\ge 1$,
 $$I^{\widehat{N}}(f)I^{\widehat{N}}(g)=\sum_{k=0}^{\min\{p, q\}}I^{\widehat{N}}(f\stackrel{k}\smallfrown g)+\sum_{k=1}^{\min\{p, q\}}I^{\widehat{N}}(f\star_{k}^{k-1} g), I^{\widehat{N}}(f)^*=I^{\widehat{N}}(f^*), \eqno (1.6)$$ (see (3) of Page 2145 in \cite{SB2}) The algebra $L^2(\mathcal{X}(\widehat{N}, \varphi))$ is called the {\sl free Poisson (multiple integral) algebra}.

\section{Multidimensional free Poisson limits on the free Wigner chaos}

 In this section, we  study the convergence of a sequence in the free Wigner chaos to a multidimensional free Poisson distribution, aiming at giving a fourth moment type theorem for the convergence.

 Multidimensional (finite dimensional) free Poisson distributions were studied in \cite{RS}.   More general (infinite dimensional) multidimensional free Poisson distributions  were defined in \cite{AG} through a multidimensional free Poisson limit theorem.
\begin{Definition}[Definition 2.7 in \cite{AG}] Let $\{\alpha_i:i=1,2,\cdots\}$ be a sequence of real numbers, $\{\lambda_i>0:i=1, 2, \cdots\}$ with $\lambda=\sup\{\lambda_i:i\ge 1\}<\infty$,  and for each $N\in \mathbb{N}, N\ge\lambda$, $(\A, \varphi_N)$ be a $C^*$-probability space. A sequence $\{b_i:i=1, 2, \cdots\}$ of random variables in a non-commutative probability space $(\B, \phi)$ has a multidimensional free Poisson distribution, if there is a family $\{p_N^{(i)}:  i\in \mathbb{N}\}$ of projections in $\A_N$ and an ultrafilter $\omega \in \beta\mathbb{N}\setminus \mathbb{N}$ such that $\varphi_N(p_N^{(i)})=\frac{\lambda_i}{N}$,  and $$\kappa_n(b_{i(1)}, b_{i(2)}, \cdots, b_{i(n)})=\alpha_{i(1)}\alpha_{i(2)}\cdots \alpha_{i(n)}\lim_{N\rightarrow \omega}N\varphi_N(p_N^{(i(1))}p_N^{(i(2))}\cdots p_N^{(i(n))}),$$
 for all  $(i(1), i(2), \cdots, i(n))\in \mathbb{N}^n.$
\end{Definition}

If we choose  $\{p_N^{(i)}: i=1, 2, \cdots\}$ to be an orthogonal sequence of projections for each $N$, then we get $\kappa_n(b_{\epsilon(1)}, \cdots, b_{\epsilon(n)})=0,$ if $\epsilon:\{1, 2, \cdots n\}\rightarrow \mathbb{N}$ is not constant, and $n\ge 2$. It follows that the sequence $\{b_n:n=1, 2, \cdots\}$ is a free family of free Poisson random variables. In other words, a free sequence $\{b_n: n=1, 2, \cdots \}$ of free Poisson random variables has a multidimensional free Poisson distribution in sense of Definition 2.1.
Let $a_i=b_i-\varphi(b_i)1, \forall i\in \mathbb{N}$, for a sequence $\{b_n:n=1, 2,\cdots \}$ having a multidimensional free Poisson distribution. We then have
$$\kappa_1(a_i)=0, i=1, 2, \cdots;$$ $$\kappa_n(a_{i(1)}, a_{i(2)}, \cdots, a_{i(n)})=\kappa_n(b_{i(1)}, \cdots, b_{i(n)}),\forall i(1), i(2), \cdots, i(n), n\in \mathbb{N}, n\ge 2.$$
We say that $\{a_i:i=1, 2, \cdots\}$ has a centered multidimensional free Poisson distribution.
\begin{Theorem} Let $\{a_n:n=1, 2, \cdots\}$ be a free sequence of    free Poisson random variables with parameters $\lambda_i, i=1, 2, \cdots$,  in a non-commutative probability space $(\A, \varphi)$. Then for $2\le n\in \mathbb{N}$ and $\chi:\{1, 2, \cdots, n\}\rightarrow \mathbb{N}$, we have $$\varphi(a_{\chi(1)}\cdots a_{\chi(n)})=\sum_{\pi\in NC(n), \pi\le \chi}\prod_{V\in \pi}\lambda_{\chi(V)}, $$ where $\chi$ in the sum subscript of the right hand side is the partition of $\{1, 2, \cdots, n\}$ defined by $l_1\sim_\chi l_2$ if and only if $\chi(l_1)=\chi(l_2)$, for $1\le l_1, l_2 \le n$, $\chi(V)$, the subscript of $\lambda$,  is the common value of $\chi$ when restricted to $V\in \pi$.
\end{Theorem}
\begin{proof} The proof is a simple application of the moments and free cumulants conversion formula ($(11.8)$ in \cite{NS}). Under the assumptions of this theorem, we have
\begin{align*}
\varphi(a_{\chi(1)} \cdots a_{\chi(n)})=&\sum_{\pi\in NC(n), \pi\le \chi}\kappa_\pi(a_{\chi(1)}, \cdots, a_{\chi(n)})\\
=&\sum_{\pi\in NC(n), \pi\le \chi}\prod_{V\in\pi}\kappa_V(a_{\chi(1)} \cdots a_{\chi(n)})\\
=&\sum_{\pi\in NC(n), \pi\le \chi}\prod_{V\in \pi}\lambda_{\chi(V)}
\end{align*}
\end{proof}

As the discussions in  the proof of Theorem 1.4 in \cite{NP}, iterative applications of the product formula $(1.3)$ yields the following formula $$I(f_n^{(i_1)}\cdots I(f_n^{(i_m)})=\sum_{(r_1, r_2, \cdots, r_{m-1})\in A_m}I((\cdots (f_n^{(i_1)})\stackrel{r_1}\smallfrown f_n^{(i_2)})\stackrel{r_2}\smallfrown f_n^{(i_3)}\cdots\stackrel{r_{m-2}}\smallfrown f_n^{(i_{m-1})})\stackrel{r_{m-1}}\smallfrown f_n^{(i_{m})}),$$
for all $i_1, \cdots, i_m, m\in \mathbb{N}, m\ge 2$, where
\begin{align*}A_m=&\{(r_1, r_2, \cdots, r_{m-1})\in \{0, 1, 2, \cdots, q\}^{m-1}:r_2\le 2q-2r_1,\\
 &r_3\le 3q-2r_1-2r_2, \cdots, r_{m-1}\le (m-1)q-2r_1-\cdots-2r_{r_{m-2}}\}.
 \end{align*}
 Then, we have $$\varphi(I(f_n^{(i_1)})\cdots I(f_n^{(i_m)}))=\sum_{r_1, r_2, \cdots, r_{m-1}\in B_m}(\cdots (f_n^{(i_1)}\stackrel{r_1}\smallfrown f_n^{(i_2)})\stackrel{r_2}\smallfrown f_n^{(i_3)}\cdots\stackrel{r_{m-2}}\smallfrown f_n^{(i_{m-1})})\stackrel{r_{m-1}}\smallfrown f_n^{(i_{m})},$$
 where $B_m=\{(r_1, \cdots, r_{m-1})\in A_m:2r_1+2r_2+\cdots+2r_{m-1}=mq\}$.

We need several technical results to prove the main result of this section.

\begin{Lemma}
Given an even number  $q\ge 2$ and $\lambda >0$, consider a bi-indexed sequence $\{f_{n}^{(i)}: n, i\in \mathbb{N}\}$ of  symmetric functions in $L^2(\mathbb{R}_+^q)$ such that $\lim_{n\rightarrow \infty}\langle f_{n} ^{(i)}, f_{n}^{(j)}\rangle_{L^2(\mathbb{R}_+^q)}=\delta_{i,j}\lambda_i$, for all $i, j \in \mathbb{N}$. As $n\rightarrow \infty$, one has that, for $i, j\in \mathbb{N}$,  $$\varphi(I(f_n^{(i)})I(f_n^{(j)})^2I(f_n^{(i)})-2 (I(f_n^{(i)})I(f_n^{(j)})^2)))\rightarrow \left\{\begin{array}{ll}\lambda_i\lambda_j, &\text{if  } i\ne j;\\
 2\lambda_i^2-\lambda_i, &\text{if  } i=j
 \end{array}\right.\eqno (2.1)$$ if and only if $\|f_{n}^{(i)}\stackrel{q/2}\smallfrown f_{n}^{(i)}-f_{n}^{(i)}\|_{L^2(\mathbb{R}_+^q)}\rightarrow 0$ and $\|f_{n}^{(i)}\stackrel{r}\smallfrown f_{n}^{(i)}\|_{L^2(\mathbb{R}_+^q)}\rightarrow 0$, for all $r=1, 2, \cdots, q-1$, $r\ne q/2$, and $i=1, 2, \cdots$; and $\|f_{n}^{(i)}\stackrel{r}\smallfrown f_{n}^{(j)}\|_{L^2(\mathbb{R}_+^q)}\rightarrow 0$, for all $r=1, 2, \cdots, q-1$,  if $i\ne j$, $i, j=1, 2, \cdots$.
\end{Lemma}
\begin{proof}
When $i=j$, this lemma is Lemma 5.1 in \cite{NP}.

Now we consider $i\ne j$. In this case,  by the product formula (1.3), we have
\begin{align*}
 &I(f_{n}^{(i)})I(f_{n}^{(j)})-I(f_{n}^{(j)})\\
 =&\langle f_n^{(i)}, f_n^{(j)}\rangle +I(f_{n}^{(i)}\stackrel{0}\smallfrown f_{n}^{(j)})+I(f_{n}^{(i)}\stackrel{q/2}\smallfrown f_{n}^{(j)}-f_{n}^{(j)})+\sum_{1\le r\le q-1, r\ne q/2}I(f_{n}^{(i)}\stackrel{r}\smallfrown f_{n}^{(j)}).
 \end{align*}
  By (1.2) and (1.0), we get
 \begin{align*}
 &\lim_{n\rightarrow\infty}\varphi((I(f_{n}^{(i)})I(f_{n}^{(j)})-I(f_{n}^{(j)}))(I(f_{n}^{(j)})I(f_{n}^{(i)})-I(f_{n}^{(j)}))\\
 =& \lambda_i\lambda_j
 +\lim_{n\rightarrow\infty}(\|f_{n}^{(i)}\stackrel{q/2}\smallfrown f_{n}^{(j)}-f_{n}^{(j)}\|^2+\sum_{1\le r<q, r\ne q/2}\|f_{n}^{(i)}\stackrel{r}\smallfrown f_{n}^{(j)}\|^2)\\
 =&\lambda_i\lambda_j+\lim_{n\rightarrow \infty} (\|f_{n}^{(i)}\stackrel{q/2}\smallfrown f_{n}^{(j)}\|^2+\|f_n^{(j)}\|^2+2\Re\langle f_{n}^{(i)}\stackrel{q/2}\smallfrown f_{n}^{(j)},f_n^{(j)} \rangle+ \sum_{1\le r<q, r\ne q/2}\|f_{n}^{(i)}\stackrel{r}\smallfrown f_{n}^{(j)}\|^2)\\
 =&\lambda_i\lambda_j+\lambda_j +\lim_{n\rightarrow \infty} (\|f_{n}^{(i)}\stackrel{q/2}\smallfrown f_{n}^{(j)}\|^2+2\Re\langle f_{n}^{(i)}\stackrel{q/2}\smallfrown f_{n}^{(j)},f_n^{(j)} \rangle +\sum_{1\le r<q, r\ne q/2}\|f_{n}^{(i)}\stackrel{r}\smallfrown f_{n}^{(j)}\|^2).\\
 \end{align*}
 On the other hand,
 \begin{align*}
 &\lim_{n\rightarrow \infty}\varphi((I(f_{n}^{(i)})I(f_{n}^{(j)})-I(f_{n}^{(j)}))(I(f_{n}^{(j)})I(f_{n}^{(i)})-I(f_{n}^{(j)}))\\
 =&\lim_{n\rightarrow \infty}(\varphi(I(f_{n}^{(i)})I(f_{n}^{(j)})^2I(f_{n}^{(i)})-2I(f_{n}^{(i)})I(f_{n}^{(j)})^2))+\lambda_j.
 \end{align*}
 We thus get that  $\lim_{n\rightarrow \infty}(\varphi(I(f_{n}^{(i)})I(f_{n}^{(j)})^2I(f_{n}^{(i)})-2I(f_{n}^{(i)})I(f_{n}^{(j)})^2))=\lambda_i\lambda_j$ if and only if $$\lim_{n\rightarrow \infty} (\|f_{n}^{(i)}\stackrel{q/2}\smallfrown f_{n}^{(j)}\|^2+2\Re\langle f_{n}^{(i)}\stackrel{q/2}\smallfrown f_{n}^{(j)},f_n^{(j)} \rangle +\sum_{1\le r<q, r\ne q/2}\|f_{n}^{(i)}\stackrel{r}\smallfrown f_{n}^{(j)}\|^2)=0.\eqno (2.2)$$
Let $\overline{s}=(s_1, \cdots, s_{q/2}), \overline{s}^r=(s_{q/2}, \cdots, s_1), d\overline{s}=ds_1\cdots ds_{q/2}$, $\overline{t_1}=(t_1, \cdots, t_{q/2}), d\overline{t_1}=dt_1 \cdots dt_{q/2}$, and $\overline{t_2}=(t_{q/2+1}, \cdots, t_q), d\overline{t_2}=dt_{q/2+1}\cdots dt_q$.  We have
\begin{align*}
\langle f_n^{(i)}\stackrel{q/2}\smallfrown f_n^{(j)}, f_n^{(j)}\rangle=&\int f_n^{(i)}(\overline{t_1}, \overline{s^r})f_n^{(j)}(\overline{s}, \overline{t_2})\overline{f_n^{(j)}}(\overline{t_1},  \overline{t_2})d\overline{s}d\overline{t_1} d\overline{t_2}\\
=&\int f_n^{(i)}(\overline{t_1}, \overline{s^r})(f_n^{(j)}\stackrel{q/2}\smallfrown f^{(j)}_n)(\overline{s},\overline{t_1^r})d\overline{s}d\overline{t_1}\\
=& \int f_n^{(i)}(\overline{t_1}, \overline{s^r})(\overline{f_n^{(j)}\stackrel{q/2}\smallfrown f^{(j)}}_n)(\overline{t_1},\overline{s^r})d\overline{s}d\overline{t_1}\\
=&\langle f^{(i)}_n, f_n^{(j)}\stackrel{q/2}\smallfrown f^{(j)}_n \rangle,
\end{align*}
where second and third equalities  hold because  $f_n^{(j)}$ and $f^{(j)}_n\stackrel{q/2}\smallfrown f_n^{(j)}$ are  symmetric, respectively.

Now suppose that $(2.1)$ holds, by the conclusion of the case that $i=j$ (see also Lemma 5.1 in \cite{NP}), $\lim_{n\rightarrow \infty }\|f_n^{(j)}\stackrel{q/2}\smallfrown f_n^{(j)}-f_n^{(j)}\|=0$. It follows that
\begin{align*}
\lim_{n\rightarrow \infty}|\langle f_n^{(i)}\stackrel{q/2}\smallfrown f_n^{(j)}, f_n^{(j)}\rangle|=&\lim_{n\rightarrow \infty}|\langle f_n^{(i)}, f_n^{(j)}\stackrel{q/2}\smallfrown f^{(j)}_n\rangle|\\
\le &\lim_{n\rightarrow \infty}|\langle f_n^{(i)}, f_n^{(j)}\rangle|+\lim_{n\rightarrow \infty}|\langle f_n^{(i)}, f_n^{(j)}\stackrel{q/2}\smallfrown f^{(j)}_n-f_n^{(j)}\rangle|=0.
\end{align*}
Combining the above discussion with $(2.2)$, we get $$\lim_{n\rightarrow \infty}\|f_n^{(i)}\stackrel{r}\smallfrown f_n^{(j)}\|=0,\eqno (2.3)$$ for $r=1, 2, \cdots, q-1$.

Conversely, if $(2.3)$ hold, we get $(2.2)$. Hence, $(2.1)$ holds for $i\ne j$.
\end{proof}
\begin{Lemma}
Let $q\ge 2$ be an even integer, and consider a bi-indexed sequence $\{f_{n}^{(i)}: n, i\in \mathbb{N}\}$ of  symmetric functions in $L^2(\mathbb{R}_+^q)$ such that $\lim_{n\rightarrow \infty}\langle f_{n}^{(i)}, f_{n}^{(j)}\rangle_{L^2(\mathbb{R}_+^q)}=\delta_{ij}\lambda_i$, $\lambda_i>0$ for all $i, j\in \mathbb{N}$.  If $\lim_{n\rightarrow \infty}\|f_{n}^{(i)}\stackrel{q/2}\smallfrown f_{n}^{(i)}-f_{n}^{(i)}\|_{L^2(\mathbb{R}_+^q)}= 0$, $\lim_{n\rightarrow \infty}\|f_n^{(i)}\stackrel{r}\smallfrown f_n^{(i)}\|=0$, for all $i=1, 2, \cdots$, and $1\le r\le q-1$ and $r\ne q/2$, and $\lim_{n\rightarrow \infty}\|f_n^{(i)}\stackrel{r}\smallfrown f_n^{(j)}\|=0$, for $r=1, 2, \cdots, q-1$, all $i\ne j$, then for $1<m\in \mathbb{N}$ and $\chi:\{1, 2, \cdots, m\}\rightarrow \mathbb{N}$, we have
 \begin{align*}&\lim_{n\rightarrow \infty}\sum_{(r_1, r_2, \cdots, r_{m-1})\in D_m}(\cdots (f_{n}^{(\chi(1))}\stackrel{r_1}\smallfrown f_{n}^{(\chi(2))})\stackrel{r_2}\smallfrown f_{n}^{(\chi(3))}\cdots\stackrel{r_{m-2}}\smallfrown f_{n}^{(\chi(m-1)})\stackrel{r_{m-1}}\smallfrown f_{n}^{(\chi(m))}\\
  =&\varphi(a_{\chi(1)}\cdots a_{\chi(m)}),
  \end{align*}
   where $D_m=B_m\cap \{0,\frac{q}{2}, q\}^{(m-1)}$, and $\{a_i:i=1, 2, \cdots\}$ is a free sequence of free Poisson random variables $a_i$ with parameters $\lambda_i>0$, $i=1, 2, \cdots$.
\end{Lemma}
\begin{proof}
We follow the ideas of the proof of Lemma 5.2 in \cite{NP}. For $r_1, \cdots r_{j}\in \{0, q/2, q\}$, let $$G_0:=f_n^{(\chi(1))},G_1:=f_n^{(\chi(1))}\stackrel{r_1}\smallfrown f_n^{(\chi(2))}, \cdots, G_j:=G_{j-1}\stackrel{r_{j}}\smallfrown f_n^{(\chi(j+1))}, j=2, 3, \cdots.$$ We show that $G_j$ is either a scalar or a multiple of an object of the type $H_1\otimes\cdots \otimes H_l$, $l\ge 1$, where $H_i=f_n^{(l)}$, or $H_i=f_n^{(s_1)}\stackrel{q/2}\smallfrown\cdots \stackrel{q/2}\smallfrown f_n^{(s_k)}$. We prove it by  induction on $j$. When $j=1$, $G_j=f_n^{(\chi(1))}\otimes f_n^{(\chi(2))}$, if $r_1=0$; $G_1=f_n^{(\chi(1))}\stackrel{q/2}\smallfrown f_n^{(\chi(2))}$, if $r_1=q/2$; $G_1=\langle f_n^{(\chi(1))},f_n^{(\chi(2))}\rangle$, is a scalar, if $r_1=q$. Suppose that $G_{j-1}$ has the desired form. We prove that so does $G_j$. Let $$G_{j-1}=H_1\otimes\cdots\otimes H_l, l\ge 0,$$ where $G_{j-1}$ is a scalar, when $l=0$. If $r_j=0$, $G_j=H_1\otimes\cdots\otimes H_l\otimes f_n^{(\chi(j+1))}$. If $r_j=q/2$, $G_j=H_1\otimes\cdots\otimes (H_{l}\stackrel{q/2}\smallfrown f_n^{(\chi(j+1))})$. If $r_j=q$, then $G_j=\langle H_l, f_n^{(\chi(j+1))}\rangle H_1\otimes\cdots\otimes H_{l-1}$. It follows that $G_j$ has the desired form in all cases.

Now we prove that $$\langle f_{n}^{(i_1)}\stackrel{q/2}\smallfrown\cdots\stackrel{q/2}\smallfrown f_{n}^{(i_k)}, f_{n}^{(i_{k+1})}\rangle\rightarrow \left\{\begin{array}{ll}\lambda_{i_1},&\text{if } i_1=i_2=\cdots=i_{k+1};\\
0,& \text{otherwise}
\end{array}\right.   \eqno (2.4)$$ as $n\rightarrow \infty$. We need a Cauchy-Schwartz type inequality. For  $f\in L^2(\mathbb{R}_+^p)$, $g\in L^2(\mathbb{R}_+^q)$, and $0\le r\le \min\{p,q\}$, we prove that $$\|f\stackrel{r}\smallfrown g\|_{\mathbb{R}_+^{p+q-2r}}^2\le \|f\|_{\mathbb{R}_+^p}^2\|g\|_{\mathbb{R}_+^q}^2.\eqno (2.5)$$
It is obvious that we need only to prove the non-trivial case $0<r<\min\{p,q\}$. We use some vector notations to simplify the expressions. Let $\overline{x}_1=(x_1, x_2, \cdots, x_{p-r}), \overline{t}^i=(t_r, t_{r-1}, \cdots, t_1), \overline{t}=(t_1, t_2, \cdots, t_r), \overline{x}_2=(x_{q-r+1}, \cdots, x_{p+q-2r})$. We then have
\begin{align*}
\|f\stackrel{r}\smallfrown g\|^2_{\mathbb{R}_+^{p+q-2r}}&=\int_{\mathbb{R}_+^{p+q-2r}}|f\stackrel{r}\smallfrown g|^2d\overline{x_1}d\overline{x}_2\\
&\le \int_{\mathbb{R}_+^{p+q-2r}}(\int_{\mathbb{R}_+^{i}}|f(\overline{x}_1, \overline{t}^i)| |g(\overline{t}, \overline{x}_2)|d\overline{t})^2d\overline{x_1}d\overline{x}_2\\
&\le  \int_{\mathbb{R}_+^{p+q-2r}}\int_{\mathbb{R}_+^{r}}|f(\overline{x}_1, \overline{t}^i)|^2d\overline{t} \int_{\mathbb{R}_+^r}|g(\overline{t}, \overline{x}_2)|^2d\overline{t}d\overline{x_1}d\overline{x}_2\\
&=\int_{\mathbb{R}_+^{p-r}}\int_{\mathbb{R}_+^{r}}|f(\overline{x}_1, \overline{t}^i)|^2d\overline{t}d\overline{x}_1 \int_{\mathbb{R}_+^{q-r}}\int_{\mathbb{R}_+^r}|g(\overline{t}, \overline{x}_2)|^2d\overline{t}d\overline{x}_2\\
&=\|f\|^2_{L^2(\mathbb{R}_+^p)}\|g\|^2_{L^2(\mathbb{R}_+^q)}.
\end{align*}

When $i_1=i_2=\cdots=i_{k+1}$, $(2.4)$ is $(5.2)$ in Lemma 5.2 in \cite{NP}. If $i_1=\cdots=i_k\ne i_{k+1}$,
we have
\begin{align*}
&\lim_{n\rightarrow \infty}|\langle f_{n}^{(i_1)}\stackrel{q/2}\smallfrown\cdots\stackrel{q/2}\smallfrown f_{n}^{(i_k)}, f_{n}^{(i_{k+1})}\rangle|\\
\le&\lim_{n\rightarrow\infty}|\langle (f_{n}^{(i_1)}\stackrel{q/2}\smallfrown f_{n}^{(i_2)}-f_{n}^{(i_2)})\stackrel{q/2}\smallfrown\cdots\stackrel{q/2}\smallfrown f_{n}^{(i_k)}, f_{n}^{(i_{k+1})}\rangle|\\
+&\lim_{n\rightarrow\infty}|\langle f_{n}^{(i_2)}\stackrel{q/2}\smallfrown f_{n}^{(i_3)}\stackrel{q/2}\smallfrown\cdots\stackrel{q/2}\smallfrown f_{n}^{(i_k)}, f_{n}^{(i_{k+1})}\rangle|\\
\le & \lim_{n\rightarrow\infty}\|f_{n}^{(i_1)}\stackrel{q/2}\smallfrown f_{n}^{(i_2)}-f_{n}^{(i_2)}\|\|f_{n}^{(i_3)}\|\cdots\|f_{n}^{(i_k)}\|\|f_{n}^{(i_{k+1})}\|\\
+&\lim_{n\rightarrow\infty}|\langle f_{n}^{(i_2)}\stackrel{q/2}\smallfrown f_{n}^{(i_3)}\stackrel{q/2}\smallfrown\cdots\stackrel{q/2}\smallfrown f_{n}^{(i_k)}, f_{n}^{(i_{k+1})}\rangle|\\
=&\lim_{n\rightarrow\infty}|\langle f_{n}^{(i_2)}\stackrel{q/2}\smallfrown f_{n}^{(i_3)}\stackrel{q/2}\smallfrown\cdots\stackrel{q/2}\smallfrown f_{n}^{(i_k)}, f_{n}^{(i_{k+1})}\rangle|\\
\le &\cdots\\
\le &\lim_{n\rightarrow \infty}|\langle f_{n}^{(i_k)}, f_{n}^{(i_{k+1})}\rangle|=0.
\end{align*}
 When $|\{i_1, \cdots, i_k\}|>1$, let $i_p$ be the first index in the ordered sequence $\{i_1, i_2, \cdots, i_k\}$ such that $i_1\ne i_p$. By $(2.5)$ and the above discussion, we get
 $$\lim_{n\rightarrow \infty}|\langle f_{n}^{(i_1)}\stackrel{q/2}\smallfrown\cdots\stackrel{q/2}\smallfrown f_{n}^{(i_k)}, f_{n}^{(i_{k+1})}\rangle|\le \lim_{n\rightarrow \infty}\|f_n^{(i_1)}\stackrel{q/2}\smallfrown f_n^{(i_p)}\|\|f_n^{(i_{p+1})}\|\cdots\|f_n^{(i_{k+1})}\|=0.$$ We've proved $(2.4)$.

For  $1<m\in \mathbb{N}$, and $\chi:\{1, 2, \cdots, m\}\rightarrow \mathbb{N}$,
we try to find $\lim_{n\rightarrow \infty}G_{m-1}$, where $$G_{m-1}:=(\cdots (f_n^{(\chi(1))}\stackrel{r_1}\smallfrown f_n^{(\chi(2))})\stackrel{r_2}\smallfrown f_n^{(\chi(3))}\cdots\stackrel{r_{m-2}}\smallfrown f_n^{(\chi(m-1))})\stackrel{r_{m-1}}\smallfrown f_n^{(\chi(m))},$$  $(r_1, r_2, \cdots, r_{m-1})\in D_m$. Since $r_1, \cdots, r_{m-1}\in D_m$, we have $r_{m-1}=q$, by the definitions of $A_m$ and $D_m$. Let $\{r_{j(1)}, r_{j(2)}, \cdots, r_{j(k-1)}, r_{m-1}\}$ be the set of all $q$'s in the sequence $r_1, r_2, \cdots, r_{m-1}$, where $j(1)<j(2)<\cdots <j(k-1)<m-1$. Let $l_1$ be the number such that $r_{j(1)-l_1}=0$, and $r_{j(1)-l_1+1}r_{j(1)-l_1+2}\cdots r_{j(1)}\ne 0$. Then $$G_{m-1}=G_{j(1)-l_1}\otimes f_{n}^{(\chi(j(1)-l_1+1))}\stackrel{q/2}\smallfrown f_n^{(\chi(j(1)-l_1+2))}\stackrel{q/2}\smallfrown \cdots \stackrel{q/2}\smallfrown f_{n}^{(\chi(j(1)))}\stackrel{q}\smallfrown f_{n}^{(\chi(j(1)+1))}\cdots,$$ if $l_1>1$; $G_{m-1}=G_{j(1)-l_1}\otimes f_n^{\chi(j(1))}\stackrel{q}\smallfrown f_n^{(\chi(j(1)+1))}\cdots$, if $l_1=1$,  where $$\alpha_1:=f_{n}^{(\chi(j(1)-l_1+1))}\stackrel{q/2}\smallfrown f_n^{(\chi(j(1)-l_1+2))}\stackrel{q/2}\smallfrown \cdots \stackrel{q/2}\smallfrown f_{n}^{(\chi(j(1)))}\stackrel{q}\smallfrown f_{n}^{(\chi(j(1)+1))},$$ $$\beta_1:=f_n^{\chi(j(1))}\stackrel{q}\smallfrown f_n^{(\chi(j(1)+1))}$$ are constants.  Let $V_1=\{j(1)-l_1+1, j(1)-l_1+2, \cdots, j(1)+1\}$.  Removing $\alpha_1$ or $\beta_1$ from $G_{m-1}$, we get a new expression $\widetilde{G}_{m-l_1}$, which has the same form as $G_{m-1}$. We continue the above procedure until we remove all factors in $G_{m-1}$.  We then get a non-crossing partition $\pi=\{V_1, V_2, \cdots, V_k\}$ in $NC_{\ge 2}(m)=\{\pi\in NC(m):|V|>1, \forall V\in \pi\}$. By the hypothesis and $(2.4)$,  $\lim_{n\rightarrow \infty}\alpha_1\beta_1\ne 0$, if and only if  $$\chi(j(1)-l_1+1)=\cdots =\chi(j(1)+1). $$ In this case, $\lim_{n\rightarrow \infty}\alpha_1=\lim_{n\rightarrow \infty}\beta_1=\lambda_{\chi(j(1))}$. It implies that  $\lim_{n\rightarrow\infty}G_{m-1}\ne 0$, if and only if  $\pi\le \chi$, that is, $\chi|_{V_i}$ must be a constant, for $i=1, 2, \cdots, k$.   Moreover, $$\lim_{n\rightarrow \infty}G_{m-1}=\prod_{V\in\pi}\lambda_{\chi(V)},$$ if $\pi\le \chi$. For $\overline{r}=(r_1, r_2, \cdots, r_{m-1})\in D_m$, the above process gives rise to a partition $\pi\in NC_{\ge 2}(m)$ denoted by $\pi(\overline{r})$. We denote the above $G_{m-1}$ by $G_{m-1}(\overline{r})$. We have proved that $$\lim_{n\rightarrow \infty}G_{m-1}(\overline{r})\ne 0$$ if and only if $\pi(\overline{r})\le \chi$.

Conversely, if $\pi\in NC_{\ge 2}(m)$ and $\pi\le \chi$, then we can arrange the blocks of $\pi$ as $V_1, V_2, \cdots, V_k$ so that the minimal element of $V_i$ is less than that of $V_j$, if $i<j$. Since $\pi$ is non-crossing, there are   interval blocks in $\pi$. Let $V_i$ be the interval block with minimum index. Let $V_i=\{p+1, p+2, \cdots, p+j\}$.   Define $r_{p+1}=r_{p+2}=\cdots=r_{p+j-2}=q/2, r_{p+j-1}=p$,if $j>2$; $r_{p+1}=q$, if $k=2$. Since $\pi\setminus \{V_i\}$ is still non-crossing, and less than $\chi|_{\{1, 2, \cdots, m\}\setminus \{p+1, \cdots, p+j\}}$, we can continue the process till we take all blocks of $\pi$. Each time when choosing $r\in \{q/2, q\}$ to link $i$ and $j$ ($i<j$) in a block, we define $r_{j-1}=r$. We have defined $r_i$'s connecting the ascending ordered set of numbers in each block of $\pi$. For the smallest number $m_i$ of block $V_i$,  we define $r_{m_i-1}=0$, if $m_i\ne 1$, for $i=1, 2, \cdots, k$.       We thus get a $r_1, \cdots, r_{m-1}\in \{0, q/2, q\}$. We now show that $$2r_1+2r_2+\cdots+2r_{m-1}=mq.$$ In fact, for an interval $V_i$,  we have $2r_{p+1}+\cdots 2r_{p+j-1}=2(|V_i|-2)\frac{q}{2}+2q=|V_i|q$. Note that every block must be an interval block in some step of the above process of definition $r_i$'s. Moreover, every non-zero $r_i$' must link two elements in a block, when the block is treated as an interval block. It follows that $2(r_1+\cdots+r_{m-1})=\sum_{V\in \pi}|V|q=mq$. We need to prove that $(r_1, r_2, \cdots, r_{m-1})\in A_m$. It is obvious that $(r_i)_{i\in V}\in A_{|V|}$, for a block $V\in \pi$. For each $r_l$, let $V_1, \cdots, V_p$ be  the blocks of $\pi$ such that $V_i':= V_i\cap \{1, 2, \cdots, l-1\}\ne \emptyset$, $i=1, 2, \cdots, p$. Then   $$lq-2\sum_{i=1}^{l-1}r_i=\sum_{i=1}^p(|V_i'|q-2\sum_{r_j \text{ connects  numbers in } V_i'}r_j)+(l-\sum_{i=1}^p|V_i'|)q\ge q \ge r_l.$$ Therefore, $\overline{r}:=(r_1, r_2, \cdots, r_{m-1})\in D_m$. Moreover, the construction of $\overline{r}$ shows that $\pi(\overline{r})=\pi$.  Since $\pi\le \chi$, by the previous proof, the express $G_{m-1}$ corresponding to $(r_1, r_2, \cdots, r_{m-1})$ has a limit of $\prod_{V\in\pi}\lambda_{\chi(V)}$.

We have established a one to one mapping from $$\mathfrak{F}:=\{\overline{r}:=(r_1, r_2, \cdots, r_{m-1})\in D_m: \lim_{n\rightarrow \infty}G_{m-1}(\overline{r})\ne 0\}$$ onto $\mathfrak{P}:=\{\pi\in NC_{\ge 2}(m):\pi\le \chi\}$.
Note that $$\lim_{n\rightarrow \infty}\sum_{(r_1, r_2, \cdots, r_{m-1})\in D_m}(\cdots (f_{n}^{(\chi(1))}\stackrel{r_1}\smallfrown f_{n}^{(\chi(2))})\stackrel{r_2}\smallfrown f_{n}^{(\chi(3))}\cdots\stackrel{r_{m-2}}\smallfrown f_{n}^{(\chi(m-1)})\stackrel{r_{m-1}}\smallfrown f_{n}^{(\chi(m))}\ne 0$$ if and only if $\mathfrak{F}\ne \emptyset$. Moreover, by Theorem 2.2, we have $\varphi(a_{\chi(1)}\cdots a_{\chi(m)})\ne 0$ if and only if $\mathfrak{P}\ne \emptyset$. Hence, $$\lim_{n\rightarrow \infty}\sum_{(r_1, r_2, \cdots, r_{m-1})\in D_m}(\cdots (f_{n}^{(\chi(1))}\stackrel{r_1}\smallfrown f_{n}^{(\chi(2))})\stackrel{r_2}\smallfrown f_{n}^{(\chi(3))}\cdots\stackrel{r_{m-2}}\smallfrown f_{n}^{(\chi(m-1)})\stackrel{r_{m-1}}\smallfrown f_{n}^{(\chi(m))}\ne 0$$ if and only if $\varphi(a_{\chi(1)}\cdots a_{\chi(m)})\ne 0$. If $\varphi(a_{\chi(1)}\cdots a_{\chi(m)})\ne 0$, we have
\begin{align*}&\lim_{n\rightarrow \infty}\sum_{(r_1, r_2, \cdots, r_{m-1})\in D_m}(\cdots (f_{n}^{(\chi(1))}\stackrel{r_1}\smallfrown f_{n}^{(\chi(2))})\stackrel{r_2}\smallfrown f_{n}^{(\chi(3))}\cdots\stackrel{r_{m-2}}\smallfrown f_{n}^{(\chi(m-1)})\stackrel{r_{m-1}}\smallfrown f_{n}^{(\chi(m))}\\
  =&\sum_{\pi\in NC_{\ge 2}(m), \pi\le \chi}\prod_{V\pi}\lambda_{\chi(V)}=\varphi(a_{\chi(1)}\cdots a_{\chi(m)}).
  \end{align*}
  \end{proof}
\begin{Remark}We want to give an example to illustrate the process of constructing $$(r_1, r_2, \cdots, r_{m-1})\in D_{m}$$ from $\pi\in NC_{\ge 2}(m)$. Consider a partition $$\pi=\{V_1=\{1,  5, 6, 9,  10\}, V_2=\{2, 3, 4\}, V_3=\{7, 8\}\}\in NC_{\ge 2}(10),$$  where $V_2$ is an interval block. Define $r_2=10/2=5, r_3=10$, since $|V_2|=3>2$. Removing $V_2$, we get a new non-crossing partition $\pi_1=\{V_1, V_3\}$ of $\{1, 5, 6, 7, 8, 9, 10\}$. $V_3$ is an interval block of $\pi_1$. Define $r_7=10$, since $|V_3|=2$. Removing $V_3$ from $\pi_1$, we get $\pi_2=\{1,  5, 6, 9,  10\}$, which is partition $1_5$ of $\{1,  5, 6, 9,  10\}$. Define $r_4=5, r_5=5, r_8=5, r_9=10$. Finally, define $r_1=r_6=0$. It is easy to verify that $(r_1, r_2, \cdots, r_9)\in D_{10}$.
\end{Remark}
\begin{Lemma}
Under the assumptions of Lemma 2.4, we have
 $$(\cdots (f_{n}^{(\chi(1))}\stackrel{r_1}\smallfrown f_{n} ^{(\chi(2))})\stackrel{r_2}\smallfrown f_{n}^{(\chi(3))}\cdots\stackrel{r_{m-2}}\smallfrown f_{n}^{(\chi(m-1))})\stackrel{r_{m-1}}\smallfrown f_{n}^{(\chi(m))}\rightarrow 0,$$ as $n\rightarrow \infty$, for $(r_1, r_2, \cdots, r_{m-1})\in E_m:=B_m\setminus D_m$, $\chi:\{1, 2, \cdots, m\}\rightarrow\mathbb{N}$, and $m>1$.
\end{Lemma}
\begin{proof}
 Let $l$ be the minimal index,  $1\le l\le m-1$, such that $r_l\ne q/2$, and $1\le r_l< q$. Let $G_0:=f_{n}^{(\chi(1))}$, $G_1=f_{n}^{(\chi(1))}\stackrel{r_1}\smallfrown f_{n}^{(\chi(2))}$, $\cdots, G_{m-1}=G_{m-2}\stackrel{r_{m-1}}\smallfrown f_{n}^{(\chi(m))}$. Since $r_i\in \{0, q/2, q\}$, for $i=1, 2, \cdots, l-1$, and $r_l\ge 1$, by the  proof of Lemma 2.4,   $G_{l-1}$ is a multiple of an object of the type $H_1\otimes\cdots \otimes H_k$, where $k\ge 1$, each $H_i$ is either equal to $f_{n}^{(j)}$ or to an iterated contraction of the type
 $$f_{n}^{(j_1)}\stackrel{q/2}\smallfrown \cdots\stackrel{q/2}\smallfrown f_{n}^{(j_s)}, $$ for some $s\ge 1$.  let $G_{l-1}=cH_1\otimes \cdots \otimes H_k$.

 Case I.   $H_k=f_{n}^{(j)}$. We then have
\begin{align*}
&(\cdots (f_{n}^{(\chi(1))}\stackrel{r_1}\smallfrown f_{n}^{(\chi(2))})\stackrel{r_2}\smallfrown f_{n}^{(\chi(3))}\cdots\stackrel{r_{m-2}}\smallfrown f_{n}^{(\chi(m-1))})\stackrel{r_{m-1}}\smallfrown f_{n}^{(\chi(m))}\\
=&(\cdots(cH_1\otimes \cdots \otimes H_{k-1})\otimes (f_{n}^{(j)}\stackrel{r_l}\smallfrown f_{n}^{(\chi(l+1))})\cdots)\stackrel{r_{m-1}}\smallfrown f_{n}^{(\chi(m))}.
\end{align*}
By (2.5) and the fact that $\lim_{n\rightarrow\infty}\|f_{n}^{(i)}\|=\sqrt{\lambda}$, for $ i\in \mathbb{N}$, we get
\begin{align*}
&|(\cdots (f_{n}{(\chi(1))}\stackrel{r_1}\smallfrown f_{n}^{(\chi(2))})\stackrel{r_2}\smallfrown f_{n}^{(\chi(3))}\cdots\stackrel{r_{m-2}}\smallfrown f_{n}^{ (\chi(m-1))})\stackrel{r_{m-1}}\smallfrown f_{n}^{(\chi(m))}|\\
=&|(\cdots(cH_1\otimes \cdots \otimes H_{k-1}\otimes (f_{n}^{(j)}\stackrel{r_l}\smallfrown f_{n}^{(\chi(l+1))})\cdots\stackrel{r_{m-1}}\smallfrown f_{n}^{(\chi(m))}|\\
\le &|c|\|H_1\|\cdots\|H_{k-1}\|\|f_{n}^{(j)}\stackrel{r_l}\smallfrown f_{n}^{(\chi(l+1))}\|\cdot\|f_{n}^{(\chi(l+2))}\|\cdots\|f_{n}^{(\chi(m))}\|\\
\rightarrow &0.
\end{align*}

Case II.  $H_k=f_{n}^{(j_1)}\stackrel{q/2}\smallfrown \cdots\stackrel{q/2}\smallfrown f_{n}^{ (j_s)}$. The proof is similar to Case I. The only difference is that, in this case, we need to prove that
$$\|(f_{n}^{ (j_1)}\stackrel{q/2}\smallfrown \cdots\stackrel{q/2}\smallfrown f_{n}^{ (j_s)})\stackrel{r_l}\smallfrown f_{n}^{( \chi(l+1))}\|\rightarrow 0,$$ as $n\rightarrow \infty$.
When $j_1=j_2=\cdots=j_s$, by the assumptions of this lemma,  (2.5), and the proof of $(2.4)$,  we have
$$\lim_{n\rightarrow\infty}\|(f_{n}^{(j_1)}\stackrel{q/2}\smallfrown \cdots\stackrel{q/2}\smallfrown f_{n}^{(j_s)})\stackrel{r_l}\smallfrown f_{n}^{(\chi(l+1))}\|
\le\lim_{n\rightarrow\infty}\|f_{n}^{(j_s)}\stackrel{r_l}\smallfrown f_{n}^{(\chi(l+1))}\|=0.$$

Finally, suppose that $j_t$ be the first number in the ordered sequence $j_1, j_2, \cdots, j_s$ such that $j_t\ne j_1$, that is, $j_1=j_2=\cdots=j_{t-1}\ne j_t$. By the above computations and $(2.5)$, we have
$$
\lim_{n\rightarrow \infty}\|(f_{n}^{ (j_1)}\stackrel{q/2}\smallfrown \cdots\stackrel{q/2}\smallfrown f_{n}^{ (j_s)})\stackrel{r_l}\smallfrown f_{n}^{( \chi(l+1))}\|
\le \lim_{n\rightarrow \infty}\|f_n^{(j_1)}\stackrel{q/2}\smallfrown f_n^{(j_t)}\|\cdot\|f_n^{(j_{t+1})}\|\cdots \|f_n^{(j_s)}\|f_n^{\chi(l+1)}\|=0.
$$
\end{proof}
We are able to  give a four-moment theorem for  the multidimensional free Poisson limit over the free Wigner chaos.
\begin{Theorem}
Let $q\ge 2$ be an even integer, $\{\alpha_i:\alpha_i\in \mathbb{R}, \alpha_i\ne 0, i=1, 2, \cdots\}$ and $\{\lambda_i:\lambda_i>0, \lambda_i\in \mathbb{R}\}$ be two sequences of numbers,  and $\{a_{i}:i=1,2,\cdots\}$ be a free sequence of centered free Poisson random variables $a_i$ with parameters $\lambda_i$ and $\alpha_i$ (that is, $\kappa_n(a_i)=\lambda_i\alpha_i^n, n\ge 2$, $\kappa_1(a_i)=0$) in a non-commutative probability space $(\A, \varphi)$. Let $\{I(f_n^{(i)}):i, n\in \mathbb{N}\}$ be a bi-indexed sequence of multiple integrals of order $q$ with respect to the free Brownian motion $S$, where $f_n^{(i)}$ is a symmetric function in $L^2(\mathbb{R}_+^q)$. Then the following two statements are equivalent.
\begin{enumerate}
\item $\{I(f_n^{(i)}): i=1, 2, \cdots\}$ converges in joint distribution to $\{a_i:i=1, 2, \cdots\}$, as $n\rightarrow \infty$;
\item The following equations
$$\lim_{n\rightarrow \infty}\langle f_n^{(i)}, f_n^{(j)}\rangle =\delta_{ij}\lambda_i\alpha_i^2, \eqno (2.6)$$
$$\lim_{n\rightarrow\infty}\varphi(I(f_n^{(i)})^4)=(2\lambda_i^2+\lambda_i)\alpha_i^4,\lim_{n\rightarrow\infty}\varphi(I(f_n^{(i)})^3)=\lambda_i\alpha_i^3,\eqno (2.7)$$
and $$\lim_{n\rightarrow\infty}\varphi(I(f_n^{(i)})^2I(f_n^{(j)})^2)=\lambda_i\lambda_j\alpha_i^2\alpha_j^2, \lim_{n\rightarrow\infty}\varphi(I(f_n^{(i)})I(f_n^{(j)})^2)=0, \text{if }i\ne j\eqno (2.8)$$ hold,  for all $i, j\in \mathbb{N}$.
\end{enumerate}
\end{Theorem}
\begin{proof}
Replacing $a_i$ by $\frac{a_i}{\alpha_i}$, and $f_n^{(i)}$ by $\frac{f_n^{(i)}}{\alpha_i}$, we can assume that $\alpha_1=\alpha_2=\cdots =1$.

$(1)\Rightarrow (2)$. Note that $\lim_{n\rightarrow \infty}\langle f_n^{(i)}, f_n^{(j)}\rangle=\lim_{n\rightarrow \infty}\varphi(I(f_n^{(i)}f_n^{(j)}))=\varphi(a_ia_j)=\delta_{ij}\lambda_i$, for all $i, j=1, 2, \cdots$.
Moreover, for $i, j\in \mathbb{N}$, we have
$$\lim_{n\rightarrow\infty}\varphi(I(f_n^{(i)})I(f_n^{(j)})^2I(f_n^{(i)}))=\varphi(a_i^2a_j^2)=\left\{\begin{array}{ll}\varphi(a_i^4)=2\lambda_i^2+\lambda_i, &\text{if  } i=j;\\
\varphi(a_i^2)\varphi(a_j^2)=\lambda_i\lambda_j, &\text{if }i\ne j.\end{array}\right.$$
$$\lim_{n\rightarrow\infty}\varphi(I(f_n^{(i)})I(f_n^{(j)})^2)=\varphi(a_ia_j^2)=\left\{\begin{array}{ll}\varphi(a_i^3)=\lambda_i, &\text{if  } i=j;\\
\varphi(a_i)\varphi(a_j^2)=0, &\text{if }i\ne j.\end{array}\right.$$

Converse, for $m>1, m\in \mathbb{N}$ and $\chi:\{1, 2, \cdots, m\}\rightarrow \mathbb{N}$, the condition (2), Theorem 2.2, Lemmas 2.3, 2,4, and 2.6 imply that
\begin{align*}
&\lim_{n\rightarrow\infty}\varphi(I(f_n^{(\chi(1))})\cdots I(f_n^{(\chi(m))}))\\
=&\lim_{n\rightarrow\infty}\sum_{(r_1, r_2, \cdots, r_{m-1})\in B_m}(\cdots (f_n^{(\chi(1))}\stackrel{r_1}\smallfrown f_n^{(\chi(i2))})\stackrel{r_2}\smallfrown f_n^{(\chi(3))}\cdots\stackrel{r_{m-2}}\smallfrown f_n^{(\chi(m-1))})\stackrel{r_{m-1}}\smallfrown f_n^{(\chi(m))}\\
=&\lim_{n\rightarrow\infty}\sum_{(r_1, r_2, \cdots, r_{m-1})\in D_m}(\cdots (f_n^{(\chi(1))}\stackrel{r_1}\smallfrown f_n^{(\chi(2))})\stackrel{r_2}\smallfrown f_n^{(\chi(3))}\cdots\stackrel{r_{m-2}}\smallfrown f_n^{(\chi(m-1))})\stackrel{r_{m-1}}\smallfrown f_n^{(\chi(m))}\\
+&\sum_{(r_1, r_2, \cdots, r_{m-1})\in E_m}\lim_{n\rightarrow\infty}(\cdots (f_n^{(\chi(1))}\stackrel{r_1}\smallfrown f_n^{(\chi(2))})\stackrel{r_2}\smallfrown f_n^{(\chi(3))}\cdots\stackrel{r_{m-2}}\smallfrown f_n^{(\chi(m-1))})\stackrel{r_{m-1}}\smallfrown f_n^{(\chi(m))}\\
=&\lim_{n\rightarrow\infty}\sum_{(r_1, r_2, \cdots, r_{m-1})\in D_m}(\cdots (f_n^{(\chi(1))}\stackrel{r_1}\smallfrown f_n^{(\chi(2))})\stackrel{r_2}\smallfrown f_n^{(\chi(3))}\cdots\stackrel{r_{m-2}}\smallfrown f_n^{(\chi(m-1))})\stackrel{r_{m-1}}\smallfrown f_n^{(\chi(m))}\\
=&\varphi(a_{\chi(1)}\cdots a_{\chi(m)}).
\end{align*}
\end{proof}
\section{Multidimensional free Poisson limits on the free Poisson algebra}
This section is devoted to proving a four-moment theorem for the convergence of a multiple integral sequence of functions in $L^2(\mathbb{R}_+^q)$ ($q\ge 2$) with respect to a centered free Poisson random measure $\widehat{N}$. We still use $I(f)$ to denote the free stochastic integral of a function $f\in L^2(\mathbb{R}_+^q)$ with respect to $\widehat{N}$. Let's introduce some notations adopted from Section 3 of \cite{SB2}.

Let $i$ and $m$ be two non-negative integers such that $0\le i\le m-1$.  We define a multiset $M_i^m=\{1, \cdots, 1, 0, \cdots, 0\}$,  where the element $1$ has  multiplicity of $i$ and the element $0$ has multiplicity of $m-1-i$. Let $\mathfrak{S}^m_i$ be the permutation group of the multiset $M_i^m$. For $\sigma\in \mathfrak{S}_i^m$,  and a sequence $\{f_n: n\in \mathbb{N}\}$ of functions in $L^2(\mathbb{R}_+^q)$, we define the following
\begin{align*}
\mathfrak{A}_m^\sigma &=\{(r_1, \cdots, r_{m-1})\in \{0, 1, \cdots, q\}^{m-1}: \sigma(1)\le r_1, \\
&\sigma(p)\le r_p\le pq+\sum_{k=1}^{p-1}(\sigma(k)-2r_k), \forall 2\le p\le m-1\},\\
\mathfrak{B}_m^\sigma &=\{(r_1, \cdots, r_{m-1})\in \mathfrak{A}_m^\sigma: 2\sum_{k=1}^{m-1}r_k=mq+\sum_{p=1}^{m-1}\sigma(p)\}.
\end{align*}

When $q>2$ is odd, we define
$$
\mathfrak{D}_m^\sigma =\{(r_1, \cdots, r_{m-1})\in \mathfrak{B}_m^\sigma\cap \{0, \frac{q+1}{2},q\}^{m-1}: \forall 1\le j\le m-1,$$
$$r_j\in \{0, q\}\Leftrightarrow \sigma(j)=0, r_j=\frac{q+1}{2}\Leftrightarrow \sigma(j)=1\},
\mathfrak{E}_m^\sigma =\mathfrak{B}_m^\sigma\setminus \mathfrak{D}_m^\sigma.$$
$$(\bigstar_{r_p}^{r_p-\sigma(p)}f_p)_{p=1}^{m-1}:=(\cdots ((f_1\star_{r_1}^{r_1-\sigma(1)}f_2)\star_{r_2}^{r_2-\sigma(2)}f_{3})\cdots \star_{r_{m-2}}^{r_{m-2}-\sigma(m-2)}f_{m-1})\star_{r_{m-1}}^{r_{m-1}-\sigma(m-1)} f_{m},$$
$$(\stackrel{r_p}\smallfrown f_p)_{p=1}^{m-1}:=(\cdots ((f_1\stackrel{r_1}\smallfrown f_2)\stackrel{r_2}\smallfrown f_3)\cdots\stackrel{r_{m-2}}\smallfrown f_{m-1})\stackrel{r_{m-1}}\smallfrown f_m.$$

In order to prove the main result in this section, we need some technical results.
\begin{Lemma} Let $q> 1$ and $m\ge 2$ be integers.  Let $\{f_{n}:1\le n\le m\}$ be a (finite) sequence of functions in $L^2(\mathbb{R}_+^q)$. Then
$$I(f_1)I(f_2)\cdots I(f_m)=\sum_{i=0}^{m-1}\sum_{\sigma\in \mathfrak{S}_i^m}\sum_{(r_1, \cdots, r_{m-1})\in \mathfrak{A}_m^\sigma}I((\bigstar_{r_p}^{r_p-\sigma(p)}f_p)_{p=1}^{m-1}). \eqno (3.1)$$
and $$\varphi(I(f_1)I(f_2)\cdots I(f_m))=\sum_{i=0}^{\lfloor(m-2)/2\rfloor}\sum_{\sigma\in \mathfrak{S}_{2i+\pi(qm)}^{m-1}}\sum_{(r_1, \cdots, r_{m-2})\in \mathfrak{B}_{m-1}^\sigma}((\bigstar_{r_p}^{r_p-\sigma(p)}f_p)_{p=1}^{m-2}\stackrel{q}\smallfrown f_m), \eqno(3.2)$$
where $\pi(n)=0$, if $n$ is even; $\pi(n)=1$, if $n$ is odd.
\end{Lemma}
The proof of the above lemma is exactly the same as those of Lemma 4.1 and Lemma 4.2 in \cite{SB2}, because the proofs of Lemma 4.1 and Lemma 4.2 in \cite{SB2} are processes of manipulating the numbers $r_i, r_i-\sigma(i)$, and the sets $\mathfrak{A}_m^\sigma$, $\mathfrak{B}_m^\sigma$, and $\mathfrak{S}_i^m$,  and independent of the function $f$.
\begin{Lemma}
Let $q> 1$ be an integer, and   $\{f_{n}^{(i)}:n, i\in \mathbb{N}\}$ be a bi-indexed sequence of symmetric functions in $L^2(\mathbb{R}_+^q)$ such that $\lim_{n\rightarrow \infty}\langle f_{n}^{(i)}, f_{n}^{(j)}\rangle=\delta_{ij}\lambda_i$, for all $i, j\in \mathbb{N}$. Then $(2.1)$ hold for all $i, j\in \mathbb{N}$, where free stochastic integrals are defined with respect to the centered free Poisson random measure,
if and only if the following conditions are satisfied.
\begin{enumerate}
\item The case that $q$ is even.
 We have $\lim_{n\rightarrow \infty}(\|f_{n}^{(i)}\stackrel{q/2}\smallfrown f_{n}^{(i)}-f_{n}^{(i)}\|+\|f_{n}^{(i)}\stackrel{r}\smallfrown f_{n}^{(i)}\|)=0$, for all $1\le r<q$ and $r\ne q/2$, and $\lim_{n\rightarrow \infty}\|f_{n}^{(i)}\star_r^{r-1} f_{n}^{(i)}\|=0$, for all $1\le r\le q$,  for $i=1, 2, \cdots$; and $\lim_{n\rightarrow \infty} \|f_{n}^{(i)}\stackrel{r}\smallfrown f_{n}^{(j)}\|= 0$, for all $1\le r<q$, and $\lim_{n\rightarrow \infty}\|f_{n}^{(i)}\star_r^{r-1} f_{n}^{(j)}\|= 0$, for all $1\le r\le q$,  if $i\ne j$.
 \item The case that $q$ is odd. We have $\lim_{n\rightarrow \infty}(\|f_{n}^{(i)}\star_{(q+1)/2}^{(q-1)/2}f_{n}^{(i)}-f_{n}^{(i)}\|+\|f_{n}^{(i)}\stackrel{r}\smallfrown f_{n}^{(i)}\|)=0$, for all $1\le r \le q-1$, and $\lim_{n\rightarrow \infty}\|f_{n}^{(i)}\star_r^{r-1} f_{n}^{(i)}\|= 0$, for all $1\le r\le q$ and $r\ne (q+1)/2$, $i=1, 2, \cdots$; $\lim_{n\rightarrow \infty}\|f_{n}^{(i)}\stackrel{r}\smallfrown f_{n}^{(j)}\|=0$, for all $1\le r \le q-1$, and $\lim_{n\rightarrow \infty}\|f_{n}^{(i)}\star_r^{r-1} f_{n}^{(j)}\|= 0$, for all $1\le r\le q$,  if $i\ne j$.
\end{enumerate}
\end{Lemma}
\begin{proof} When $i=j$, this lemma is Lemma 4.3 in \cite{SB2}.

Now we prove this lemma when $i\ne j$.
The proof is very similar to that of Lemma 2.4. But, in the present case, we use the production formula (1.6). When $q$ is even, the proof is almost as same as that of Lemma 2.4.

When $q$ is odd, we have
\begin{align*}
 I(f_{n}^{(i)})I(f_{n}^{(j)})-I(f_{n}^{(j)})&=\langle f_n^{(i)}, f_n^{(j)}\rangle +I(f_{n}^{(i)}\stackrel{0}\smallfrown f_{n}^{(j)})+I(f_{n}^{(i)}\star_{(q+1)/2}^{(q-1)/2} f_{n}^{(j)}-f_{n}^{(j)})\\
 &+\sum_{1\le r\le q-1}I(f_{n}^{(i)}\stackrel{r}\smallfrown f_{n}^{(j)})
 + \sum_{1\le k\le q, k\ne (q+1)/2}I(f_{n}^{(i)}\star_k^{k-1} f_{n}^{(j)}).
 \end{align*}

 Using the fact that multiple free Poisson integrals of different orders are orthogonal in $L^2(\A,\varphi)$ (see (1.5)), we get
 \begin{align*}
 &\lim_{n\rightarrow \infty}\varphi((I(f_{n}^{(i)})I(f_{n}^{(j)})-I(f_{n}^{(j)}))(I(f_{n}^{(j)})I(f_{n}^{(i)})-I(f_{n}^{(j)}))\\
 =& \lambda_i\lambda_j+\lambda_j+\lim_{n\rightarrow\infty}(\|f_{n}^{(i)}\star_{(q+1)/2}^{(q-1)/2} f_{n}^{(j)}\|^2-2\Re \langle f_{n}^{(i)}\star_{(q+1)/2}^{(q-1)/2} f_{n}^{(j)}, f_{n}^{(j)}\rangle)\\
 +&\lim_{n\rightarrow \infty}(\sum_{1\le r<q}\|f_{n}^{(i)}\stackrel{r}\smallfrown f_{n}^{(j)}\|^2+\sum_{1\le k\le q, k\ne (q+1)/2}\|f_{n}^{(i)}\star_k^{k-1} f_{n}^{(j)}\|^2).
 \end{align*}
 On the other hand,
 $$\varphi((I(f_{n,i})I(f_{n,j})-I(f_{n,j}))(I(f_{n,j})I(f_{n,i})-I(f_{n,j}))=\varphi(I(f_{n,i})I(f_{n,j})^2I(f_{n,i})-2I(f_{n,i})I(f_{n,j})^2)+\lambda_j.$$
 We thus get a statement similar to that involving $(2.2)$: $$\lim_{n\rightarrow \infty}(\varphi(I(f_{n}^{(i)})I(f_{n}^{(j)})^2I(f_{n}^{(i)})-2I(f_{n}^{(i)})I(f_{n}^{(j)})^2))=\lambda_i\lambda_j$$ if and only if $$\lim_{n\rightarrow\infty}(\|f_{n}^{(i)}\star_{(q+1)/2}^{(q-1)/2} f_{n}^{(j)}\|^2-2\Re \langle f_{n}^{(i)}\star_{(q+1)/2}^{(q-1)/2} f_{n}^{(j)}, f_{n}^{(j)}\rangle $$
 $$+\sum_{1\le r<q}\|f_{n}^{(i)}\stackrel{r}\smallfrown f_{n}^{(j)}\|^2+\sum_{1\le k\le q, k\ne (q+1)/2}\|f_{n}^{(i)}\star_k^{k-1} f_{n}^{(j)}\|^2)=0. \eqno (3.3)$$

 As in the proof of Lemma 2.4, we can prove (with the same methods as those in the corresponding proofs in Lemma 2.4) that $\langle f\star_{(q+1)/2}^{(q-1)/2}g, g\rangle=\langle f, g\star_{(q+1)/2}^{(q-1)/2}g\rangle$, for $f, g\in L^2(\mathbb{R}_+^q)$,  and therefore, $$\lim_{n\rightarrow\infty}|\langle f_n^{(i)}\star_{(q+1)/2}^{(q-1)/2}g_n^{(j)}, g_n^{(j)}\rangle|=\lim_{n\rightarrow\infty}|\langle f_n^{(i)},g_n^{(j)}\star_{(q+1)/2}^{(q-1)/2}g_n^{(j)} \rangle|=0,$$ if $(2.1)$ holds in the present situation. It implies that if $(2.1)$ holds, then $$\lim_{n\rightarrow \infty}\|f_{n}^{(i)}\stackrel{r}\smallfrown f_{n}^{(j)}\|=0, \forall 1\le r \le q-1; \lim_{n\rightarrow \infty}\|f_{n}^{(i)}\star_t^{t-1} f_{n}^{(j)}\|= 0, \forall 1\le t\le q. \eqno (3.4)$$

 Conversely, $(3.4)$ implies $(3.3)$. Therefore, $(2.1)$ holds in the case that  $q$ is odd and $i\ne j$.
\end{proof}
Now we are ready to  give  the main result of this section.
\begin{Theorem}
Let $q\ge 2$ be an even integer, $\{\alpha_i:\alpha_i\in \mathbb{R}, \alpha_i\ne 0, i=1, 2, \cdots\}$ and $\{\lambda_i:\lambda_i>0, \lambda_i\in \mathbb{R}\}$ be two sequences of numbers,  and $\{a_{i}:i=1,2,\cdots\}$ be a free sequence of centered free Poisson random variables $a_i$ with parameters $\lambda_i$ and $\alpha_i$ in a non-commutative probability space $(\A, \varphi)$. Let $\{I(f_n^{(i)}):i, n\in \mathbb{N}\}$ be a bi-indexed sequence of multiple integrals of order $q$ with respect to the centred free Poisson random measure $\widehat{N}$, where $f_n^{(i)}$ is a symmetric function in $L^2(\mathbb{R}_+^q)$ such that there are numbers $M_i$ and $L_i$, and a $q$-dimensional square  $S_{n,i}$ such that $$\sup\{|f_{n}^{(i)}(t_1, \cdots, t_q)|:(t_1, \cdots, t_q)\in \mathbb{R}_+^q, n\in \mathbb{N}\}\le M_i<\infty, $$ and
  $$\{x\in \mathbb{R}_+^{q}:f_{n}^{(i)}(x)\ne 0\} \subset S_{n,i}, \sup\{\mu (S_{n,i}):n=1, 2, \cdots\}\le L_i, $$ for $i\in \mathbb{N}$. Then the following two statements are equivalent.
\begin{enumerate}
\item $\{I(f_n^{(i)}): i=1, 2, \cdots\}$ converges in  distribution to $\{a_i:i=1, 2, \cdots\}$, as $n\rightarrow \infty$;
\item The equations $(2.6)$, $(2.7)$, and $(2.8)$ hold,  for all $i, j\in \mathbb{N}$, where the free integrals in $(2.7)$ and $(2.8)$ are defined with respect to the centered free Poisson random measure.
\end{enumerate}
\end{Theorem}
\begin{proof}
As the proof of Theorem 2.7, replacing $a_i$ by $\frac{a_i}{\alpha_i}$, and $f_n^{(i)}$ by $\frac{f_n^{(i)}}{\alpha_i}$, we can assume that $\alpha_1=\alpha_2=\cdots =1$.

The implication of $(1)\Rightarrow (2)$ is same as the proof of $(1)\Rightarrow (2)$ of Theorem 2.7.

Conversely, for $\chi(1), \cdots, \chi(m), 1<m\in \mathbb{N}$, by $(3.2)$, we have
\begin{align*}
&\varphi(I(f_{n}^{(\chi(1))})I(f_{n}^{(\chi(2))})\cdots I(f_{n}^{(\chi(m))}))\\
=&\sum_{(r_1, \cdots, r_{m-2})\in B_{m-1}}(\cdots((f_{n}^{(\chi(1))}\stackrel{r_1}\smallfrown f_{n}^{(\chi(2))})\stackrel{r_2}\smallfrown f_{n}^{(\chi(3))})\cdots\stackrel{r_{m-1}}\smallfrown f_{n}^{(\chi(m-1))})\stackrel{q}\smallfrown f_{n}^{(\chi(m))}\\
+&\sum_{i=1}^{\lfloor(m-2)/2\rfloor}\sum_{\sigma\in \mathfrak{S}_{2i}^{m-1}}\sum_{(r_1, \cdots, r_{m-2})\in \mathfrak{B}_{m-1}^\sigma}((\bigstar_{r_p}^{r_p-\sigma(p)}f_{n}^{(\chi(p))})_{p=1}^{m-2}\stackrel{q}\smallfrown f_{n}^{(\chi(m))}).
\end{align*}
By the proofs of Lemmas 2.4 and 2.6, the first sum on the right-hand side of the above equation converges to $\varphi(a_{\chi(1)}a_{\chi(2)}\cdots a_{\chi(m)})$, as $n\rightarrow\infty$. Note also that the second sum in the above equation disappears  when $m=2$.  Therefore, it remains to prove $$\lim_{n\rightarrow\infty}\sum_{i=1}^{\lfloor(m-2)/2\rfloor}\sum_{\sigma\in \mathfrak{S}_{2i+\pi(qm)}^{m-1}}\sum_{(r_1, \cdots, r_{m-2})\in \mathfrak{B}_{m-1}^\sigma}(\bigstar_{r_p}^{r_p-\sigma(p)}f_{n}^{(\chi(p))})_{p=1}^{m-2}\stackrel{q}\smallfrown f_{n}^{(\chi(m))}=0, m>2. \eqno (3.5)$$

We need a Cauchy-Schwartz type inequality for $f\star_{r}^{r-1}g$, where $1<r\le q$, $f, g\in L^2(\mathbb{R}_+^q)$ satisfy the bounded conditions presented in this theorem. Assume that $$\{|f(x)|:x\in \mathbb{R}_+^{q}\}\le M(f), \{|g(x)|:x\in \mathbb{R}_+^{q}\}\le M(g),$$ $$supp\{f\}\subseteq S(f), supp\{g\}\subseteq S(g),$$ where $supp\{f\}$ is the support of $f$, $S(f)$ and $S(g)$ are  $q$-dimensional squares, and $\mu(S(f))$ is the volume of $S(f)$. We have
\begin{align*}
&\|f\star_{r}^{r-1}g\|_{L^2}^2\\
=&\int_{\mathbb{R}_+^{q}}|\int_{\mathbb{R}_+^{r-1}}f(t_1, \cdots, t_{q-r}, t_{q-r+1}, s_{r-1}, \cdots, s_1)\\
\times &g(s_1, \cdots, s_{r-1}, t_{q-r+1}, t_{q-r+2}, \cdots, t_{2q-2r+1})ds_1 \cdots ds_{r-1}|^{2}dt_1\cdots dt_{2q-2r+1}\\
\le &\int_{\mathbb{R}_+}(\int_{\mathbb{R}_+^{q-1}}|f(t_1, \cdots, t_{q-r}, t_{(q-r+1}, s_{r-1}, \cdots, s_1)|^2ds_1\cdots ds_{r-1}dt_1\cdots dt_{q-r}\\
\times &\int_{\mathbb{R}_+^{q-1}}|g(s_1, \cdots, s_{r-1}, t_{q-r+1}, t_{q-r+2}, \cdots, t_{2q-2r+1})|^2ds_1 \cdots ds_{r-1}dt_{q-r+2}\cdots dt_{2q-2r+1})dt_{q-r+1}\\
\le & \int_{\mathbb{R}_+}(M(f)^2\int_{S(f)}ds_1\cdots ds_{r-1}dt_1\cdots dt_{q-r}\\
\times &\int_{\mathbb{R}_+^{q-1}}|g(s_1, \cdots, s_{r-1}, t_{q-r+1}, t_{q-r+2}, \cdots, t_{2q-2r+1})|^2ds_1 \cdots ds_{r-1}dt_{q-r+1}\cdots dt_{2q-2r+1})\\
\le & M(f)^2 (\mu(S(f)))^{(q-1)/q}\|g\|_{L^2}^2
\end{align*}
We, therefore, get the following inequality
$$\|f\star_{r}^{r-1} g\|_{L^2}\le \sqrt{M(f)M(g)(\mu(S(f))\mu(S(g)))^{(q-1)/(2q)}\|f\|_{L^2}\|g\|_{L^2}}. $$
  Let $$\beta(f,g)=\max\{\sqrt{M(f)M(g)(\mu(S(f))\mu(S(g)))^{(q-1)/(2q)}\|f\|\|g\|}, \|f\|\|g\|\}.$$ Then we have $$\|f\star_r^{r-\sigma}g\|\le \beta(f,g),$$ where $\sigma=0$ or $1$, and $f,g\in L^2(\mathbb{R}_+^q)$ satisfying the above bounded conditions. For two sequences $\{f_n:n=1, 2, \cdots\}$ and $\{g_n:n=1, 2, \cdots\}$ with uniform (square) supports
  $$supp\{f_n\}\subset S(f), supp\{g_n\}\subset S(g), n=1, 2, \cdots,$$ we have $$\lim_{n\rightarrow \infty}\beta(f_n, g_n)=0,$$ whenever $\lim_{n\rightarrow \infty}\|f_n\|=0$ and $\sup\{\|g_n\|:n=1, 2, \cdots\}<\infty$.
   Moreover, the above calculation also shows that
  $$|f\star_r^{r-\sigma}g(t_1, \dots, t_{2q-2r+\sigma})|\le M(f)M(g)\max\{S(f)^{1/q}, S(g)^{1/q}\}^{r-\sigma}, $$
  for all $(t_1,\cdots, t_{2q-2r+\sigma})\in \mathbb{R}_+^{2q-2r+\sigma}$, and
  $$\{(t_1,\cdots,  t_{2q-2r+\sigma})\in \mathbb{R}_+^{2q-2r+\sigma}: f\star_r^{r-\sigma}g(t_1,\cdots,  t_{2q-2r+\sigma})\ne 0\}\subseteq S(f)\cup S(g).$$

Let $G_{0,n}=f_n^{(\chi(1))}$, $$G_{i,n} =G_{i-1,n}\star_{r_{i}}^{r_{i}-\sigma(i)}f_n^{(\chi(i+1))}, i=1, 2, \cdots, m-2, G_{m-1,n}=G_{m-2, n}\stackrel{q}\smallfrown f_n^{(\chi(m))}.$$ It implies that
$$\lim_{n\rightarrow\infty}\|G_{m-1,n}\|\le \lim_{n\rightarrow \infty}\beta(G_{m-2, n}, f_n^{(\chi(m))})=0, \eqno (3.6)$$
whenever $\lim_{n\rightarrow \infty}\|G_{i,n}\|=0$, for some $1\le i\le m-2$. Also $(3.6)$ implies $(3.5)$.

Therefore, it is sufficient to prove $\lim_{n\rightarrow \infty}\|G_{i,n}\|=0$, for some $1\le i\le m-2$.

There exist $k$ less that $m$ and $l$ less that $m-1$, such that $\chi(1)=\cdots=\chi(k)\ne \chi(k+1)$, and $\sigma(1)=\cdots=\sigma(l-1)=0, \sigma(l)=1$.

 Case I. $l\le k$.   If there is an item $r_{l_1}$ in the sequence $r_1, r_2, \cdots r_{l-1}$ such that $r_1, \cdots, r_{l_1-1}\in \{0, q/2, q\} $ and  $r_{l_1}$ is not in $\{0, q/2, q\}$, then by the proof of Lemma 2.4, $G_{l_1-1, n}=cH_1\otimes \cdots \otimes H_v$, where $c$ is a constant,  $H_v=f_n^{(s_1)}\stackrel{q/2}\smallfrown \cdots \stackrel{q/2}\smallfrown f_n^{(s_t)}$, $H_v=f_n^{(s_1)}$ if $t=1$, other $H_j$'s have the same form as $H_v$.  By Lemma 3.2 and the proof of Lemma 2.4, we get
 \begin{align*}
  &\lim_{n\rightarrow \infty}\|G_{l_1, n}\|=\lim_{n\rightarrow \infty}\|f_n^{(\chi(1))}\stackrel{r_1}\smallfrown f_n^{(\chi(2))}\stackrel{r_2}\smallfrown f_n^{(\chi(3))}\cdots\stackrel{r_{l_1}}\smallfrown f_n^{(\chi(l_1+1))}\|\\
  \le & \lim_{n \rightarrow \infty}|c|\|H_1\|\cdots\|H_{v-1}\|\|f_n^{(s_t)}\stackrel{r_{l_1}}\smallfrown f_n^{(\chi(l_1+1))}\|=0,
\end{align*}
 since $\lim_{n\rightarrow \infty}\|f_n^{(s_{t})}\stackrel{r_{l_1}}\smallfrown f_n^{(\chi(l_1+1))}\|=0$, by Lemma 3.2. Otherwise, if $r_1,\cdots,r_{l-1}\in \{0, q/2, q\}$, the above calculation shows that
$$\lim_{n\rightarrow \infty}\|G_{l, n}\|\le \lim_{n\rightarrow\infty}|c'|\|H_1'\|\cdots\|H_{v'-1}'\|f_n^{(s_{t'}')}\star_{r_l}^{r_l-1}f_n^{(\chi(l+1))}\|=0,$$ since $\lim_{n\rightarrow\infty}\|f_n^{(s'_{t'})}\star_{r_l}^{r_l-1}f_n^{(\chi(l+1))}\|=0$ by Lemma 3.2, no matter whether $s_{t'}'=\chi(l+1)$ or not, where we assume $G_{l-1,n}=c'H_1'\otimes \cdots \otimes H_{v'}'$, and $H_{v'}'=f_n^{(s_1')}\stackrel{q/2}\smallfrown \cdots \stackrel{q/2}\smallfrown f_n^{(s_{t'}')}$.

Case II.  $k<l$.
We show that $$\lim_{n\rightarrow\infty}\|f_n^{(\chi(1))}\stackrel{r_1}\smallfrown f_n^{(\chi(2))}\cdots \stackrel{r_{k}}\smallfrown f_n^{(\chi(k+1))}\|=0. \eqno (3.7)$$  By the proof of Lemma 2.4,
$G_{k,n}=f_n^{(\chi(1))}\stackrel{r_1}\smallfrown f_n^{(\chi(2))}\cdots \stackrel{r_{k}}\smallfrown f_n^{(\chi(k+1))}=H_1\otimes \cdots \otimes H_{p}$, where each factor $H_i$ has a similar form to $G_{k,n}$, but all $r_j$'s are nonzero. Let $f_n^{(\chi(k+1))}$ be a factor of $H_p$.  Let $H_p=f_n^{(j_1)}\stackrel{s_1}\smallfrown f_n^{(j_2)}\cdots \stackrel{s_t}\smallfrown f_n^{(\chi(k+1))}$. If there is a number $s_u$ in $\{s_1, \cdots, s_{t-1}\}$, which is neither $q/2$ nor $q$, but, $s_1, \cdots, s_{u-1}\in \{q/2, q\} $, then by the proof of Case I of this theorem,  we have $$\lim_{n\rightarrow \infty}\|f_n^{(j_1)}\stackrel{s_1}\smallfrown f_n^{(j_2)}\cdots \stackrel{s_{u}}\smallfrown f_n^{(j_{u+1})}\|=0.$$
 If $s_1, \cdots, s_{t}\in\{q/2, q\}$, we get $\lim_{n\rightarrow\infty}\|H_p\|\le C_1\lim_{n\rightarrow\infty}\|f_n^{(j_{t})}\stackrel{s_t}\smallfrown f_n^{(\chi(k+1))}\|=0$, by Lemma 3.2, since $\chi(k+1)\ne j_t$, where $C_1$ is a constant. Equation $(3.7)$ follows now. 
\end{proof}
\section{Another case of infinite dimensional free Poisson distribution limits}
When  $\lambda_1=\lambda_2=\cdots=\lambda>0$, and $p_N^{(1)}=p_N^{(2)}=\cdots =p_N$ with $\varphi_N(p_N)=\lambda/N$, for $N\in \mathbb{N}, N>\lambda$, in Definition 2.1, we shall get a special kind of infinite dimensional free Poisson distributions as follows.
\begin{Definition}
A sequence $\{b_i:i=1, 2, \cdots\}$ of random variables in a non-commutative probability space $(\A, \varphi)$ has an infinite dimensional free Poisson distribution with parameters $\lambda$ and $\alpha:=(\alpha_1, \alpha_2, \cdots, )$, a sequence of non-zero real numbers, if $$\kappa_n(b_{\chi(1)}, \cdots, b_{\chi(n)})=\lambda\alpha_{\chi(1)}\cdots\alpha_{\chi(n)},$$ for $ n\in \mathbb{N}$ and $\chi:\{1,2,\cdots, n\}\rightarrow \mathbb{N}$. Let $a_i=b_i-\varphi(b_i)$, for $i=1, 2, \cdots$. Then
$$\kappa_n(a_{\chi(1)}, \cdots, a_{\chi(n)})=\kappa_n(b_{\chi(1)}, \cdots, b_{\chi(n)}),\forall n\ge 2,  \kappa(a_i)=0, i=1, 2, \cdots.$$ We say that $\{a_n:n=1, 2, \cdots\}$ has a centered infinite dimensional free Poisson distribution denoted by $Z(\lambda, \alpha)$.
\end{Definition}
In this section, we shall give four-moment theorems for convergence to a centered infinite dimensional free Poisson distribution $Z(\lambda, \alpha)$ of a bi-index $\{I(f_{n}^{(i)}): i, n=1, 2, \cdots\}$ of free stochastic integrals of $f_n^{(i)}\in L^2(\mathbb{R}_+^q)$ with respect to free Brownian motion or centered free Poisson random measure ($q>1$).

After coefficient adjusting, we can assume $\alpha_1=\alpha_2=\cdots=1$. In this case, $Z(\lambda, \alpha)$ is denoted by $Z(\lambda)$. If a sequence $\{a_i: i=1, 2, \cdots\}$ of random variables has a distribution $Z(\lambda)$, then its distribution is  same as the distribution of a single free Poisson random variable. A necessary condition for a bi-indexed sequence $\{I(f_n^{(i)}):n, i=1, 2, \cdots\}$ to converge to $\{a_i: i=1, 2, \cdots\}$ is the following asymptotically linear dependence $(4.1)$, where $f_n^{(i)}, i, n\in \mathbb{N}$, are symmetric functions in $L^2(\mathbb{R}_+^q)$, and the free stochastic integrals $I(f_n^{(i)})$ are defined with respect to the free Brownian motion or a centered free random measure.    In certain sense, $(4.1)$ indicates that limit behavior of the bi-indexed sequence  is like the limit behavior of a single-indexed  sequence. Following the strategy for the proofs of main results in previous two sections, we find that the present situation is, in certain sense,  like a special case $i=j$ in Theorems 2.7 and 3.3.    Therefore,  we can adopt  most of the ideas in the proofs in Sections 2 and 3, and in the four-moment theorem results for  single-indexed sequences of free stochastic integrals ( Theorem 1.4 in \cite{NP} and Theorem 1.5 in \cite{SB2}).

\begin{Theorem} Let $\lambda>0$ and $\alpha:=\{\alpha_1, \alpha_2, \cdots\}$ be a sequence of non-zero real numbers, and $\{f_n^{(i)}: n, i\in \mathbb{N}\}$ be a bi-index sequence of symmetric functions in $L^2(\mathbb{R}_+^q)$, where $q\in \mathbb{N}$ is even. Let $I(f)$ be the free stochastic integral of $f\in L^2(\mathbb{R}_+^q)$ with respect to the free Brownian motion.    Then the following two statements are equivalent.
\begin{enumerate}
\item The sequence $\{\{I(f_n^{(i)}):i=1, 2, \cdots\};n=1, 2, \cdots\}$ converges in distribution to $\{a_i:i=1, 2, \cdots\}$ with distribution $Z(\lambda, \alpha)$.
\item The following equations hold $$\lim_{n\rightarrow\infty}\langle f_n^{(i)}, f_n^{(j)}\rangle=\alpha_i\alpha_j\lambda,\eqno (4.1)$$   $$\lim_{n\rightarrow\infty}\varphi(I(f_n^{(i)})^2I(f_n^{(j)})^2)=(2\lambda^2+\lambda)\alpha_i^2\alpha_j^2,
    \lim_{n\rightarrow\infty}\varphi(I(f_n^{(i)})I(f_n^{(j)})^2)=\lambda\alpha_i\alpha_j^2,\eqno (4.2)$$
      for all $i, j\in \mathbb{N}$.
\end{enumerate}
\end{Theorem}
\begin{proof}
As the proofs of Theorem 2.7 and 3.3,  we only need to prove $(2)\Rightarrow (1)$ when $\alpha_1=\alpha_2=\cdots=1$.  Let $1<m\in \mathbb{N}$ and $\chi:\{1, 2, \cdots, m\}\rightarrow \mathbb{N}$. By the proof Lemma 2.3 (or See Lemma 5.1 in \cite{NP}), Condition $(2)$ implies that $$\lim_{n\rightarrow\infty}(\|f_n^{(i)}\stackrel{q/2}\smallfrown f_n^{(j)}-f_n^{(j)}\|+\|f_n^{(i)}\stackrel{r}\smallfrown f_n^{(j)}\|)=0,$$ for $r=1, 2, \cdots, q-1$ and $r\ne q/2$. Thus, by the proof of $(2.4)$, we have $$\lim_{n\rightarrow\infty}\langle f_n^{(i_1)}\stackrel{q/2}\smallfrown f_n^{(i_2)}\stackrel{q/2}\smallfrown\cdots \stackrel{q/2}\smallfrown f_n^{i_k}, f_n^{i_{k+1}}\rangle=\lambda, \eqno (4.3) $$ since $\lim_{n\rightarrow \infty}\langle f_n^{(i)}, f_n^{(j)}\rangle=\lambda$ and $\lim_{n\rightarrow\infty}\|f_n^{(i)}\stackrel{q/2}\smallfrown f_n^{(j)}-f_n^{(j)}\|=0$. We thus get the following equation ($(5.2)$ in \cite{NP}) from $(4.3)$
$$\lim_{n\rightarrow \infty}(\cdots ((f_n^{(\chi(1))}\stackrel{r_1}\smallfrown f_n^{(\chi(2))})\stackrel{r_2}\smallfrown f_n^{(\chi(3))})\cdots\stackrel{r_{m-2}}\smallfrown f_n^{(\chi(m-1))})\stackrel{r_{m-1}}\smallfrown f_n^{(\chi(m))}=\lambda^j,$$ where $(r_1, \cdots, r_{m-1})\in D_m$, and $j$ is the number of $q$'s in the sequence $r_1, \cdots, r_{m-1}$. It implies from the proof of Lemma 5.2 in \cite{NP} that
\begin{align*}&\lim_{n\rightarrow \infty}\sum_{(r_1, \cdots, r_{m-1})\in D_m}(\cdots ((f_n^{(\chi(1))}\stackrel{r_1}\smallfrown f_n^{(\chi(2))})\stackrel{r_2}\smallfrown f_n^{(\chi(3))})\cdots\stackrel{r_{m-2}}\smallfrown f_n^{(\chi(m-1))})\stackrel{r_{m-1}}\smallfrown f_n^{(\chi(m))}\\
=&\varphi(a_{\chi(1)}\cdots a_{\chi(m)}).
\end{align*}
By the proof of Lemma 2.6, we have
$$\lim_{n\rightarrow \infty}(\cdots ((f_n^{(\chi(1))}\stackrel{r_1}\smallfrown f_n^{(\chi(2))})\stackrel{r_2}\smallfrown f_n^{(\chi(3))})\cdots\stackrel{r_{m-2}}\smallfrown f_n^{(\chi(m-1))})\stackrel{r_{m-1}}\smallfrown f_n^{(\chi(m))}=0,$$ for every $(r_1, \cdots, r_{m-1})\in E_m:=B_m\setminus D_m$.
It follows that
 \begin{align*}
&\lim_{n\rightarrow\infty}\varphi(I(f_n^{(\chi(1))})\cdots I(f_n^{(\chi(m))}))\\
=&\lim_{n\rightarrow\infty}\sum_{(r_1, r_2, \cdots, r_{m-1})\in B_m}(\cdots (f_n^{(\chi(1))}\stackrel{r_1}\smallfrown f_n^{(\chi(i2))})\stackrel{r_2}\smallfrown f_n^{(\chi(3))}\cdots\stackrel{r_{m-2}}\smallfrown f_n^{(\chi(m-1))})\stackrel{r_{m-1}}\smallfrown f_n^{(\chi(m))}\\
=&\lim_{n\rightarrow\infty}\sum_{(r_1, r_2, \cdots, r_{m-1})\in D_m}(\cdots (f_n^{(\chi(1))}\stackrel{r_1}\smallfrown f_n^{(\chi(2))})\stackrel{r_2}\smallfrown f_n^{(\chi(3))}\cdots\stackrel{r_{m-2}}\smallfrown f_n^{(\chi(m-1))})\stackrel{r_{m-1}}\smallfrown f_n^{(\chi(m))}\\
+&\sum_{(r_1, r_2, \cdots, r_{m-1})\in E_m}\lim_{n\rightarrow\infty}(\cdots (f_n^{(\chi(1))}\stackrel{r_1}\smallfrown f_n^{(\chi(2))})\stackrel{r_2}\smallfrown f_n^{(\chi(3))}\cdots\stackrel{r_{m-2}}\smallfrown f_n^{(\chi(m-1))})\stackrel{r_{m-1}}\smallfrown f_n^{(\chi(m))}\\
=&\lim_{n\rightarrow\infty}\sum_{(r_1, r_2, \cdots, r_{m-1})\in D_m}(\cdots (f_n^{(\chi(1))}\stackrel{r_1}\smallfrown f_n^{(\chi(2))})\stackrel{r_2}\smallfrown f_n^{(\chi(3))}\cdots\stackrel{r_{m-2}}\smallfrown f_n^{(\chi(m-1))})\stackrel{r_{m-1}}\smallfrown f_n^{(\chi(m))}\\
=&\varphi(a_{\chi(1)}\cdots a_{\chi(m)}).
\end{align*}
\end{proof}

\begin{Lemma} Let $q>1$ be an odd number and $m\ge 2$ be an integer. Let $\{f_{n,i}:i, n\in \mathbb{N}\}$ be a bi-indexed sequence of  symmetric functions in $L^2(\mathbb{R}_+^q)$ such that $\lim_{n\rightarrow\infty}\langle f_{n,i}, f_{n,j}\rangle=\lambda>0$, for all $i, j\in \mathbb{N}$. There exist numbers $M_i$ and $L_i$, and a $q$-dimensional square  $S_{n,i}$ such that $$\sup\{|f_{n,i}(x)|:x\in \mathbb{R}_+^q, n\in \mathbb{N}\}\le M_i,$$    $$ \{x\in \mathbb{R}_+^{q}:f_{n,i}(x)\ne 0\} \subset S_{n,i}, \sup\{\mu (S_{n,i}):n=1, 2, \cdots\}\le L_i, $$ for $i\in \mathbb{N}$. If $\lim_{n\rightarrow\infty}\|f_{n,i}\star_{(q+1)/2}^{(q-1)/2}f_{n,j}-f_{n,j}\|=0$, for all $i, j\in \mathbb{N}$, then $$\sum_{i=0}^{\lfloor (m-2)/2\rfloor}\sum_{\sigma\in \mathfrak{S}_{2i+\pi(mq)}^{m-1}}\sum_{(r_1, \cdots, r_{m-2})\in \mathfrak{D}_{m-1}^\sigma}(\bigstar_{r_p}^{r_p-\sigma(p)}f_{n,p})_{p=1}^{m-2}\stackrel{q}\smallfrown f_{n,m}\rightarrow \varphi(a_{1}\cdots a_{m}),$$ as $n\rightarrow \infty$, where   $\pi(n)=0$, if $n$ is even; $\pi(n)=1$ if $n$ is odd, $\{a_i: i=1, 2, \cdots\}$ is a sequence of random variables with distribution $Z(\lambda)$.
\end{Lemma}
\begin{proof}
The proof is similar to that of Lemma 2.4. We first prove that $$G_{m-1}:=(\bigstar_{r_p}^{r_p-\sigma(p)}f_{n,p})_{p=1}^{m-2}\stackrel{q}\smallfrown f_{n,m}\rightarrow\lambda^j,\eqno (4.4)$$ as $n\rightarrow \infty$, where $(r_1, \cdots, r_{m-2})\in \mathfrak{D}_{m-1}^{\sigma}$, and $j$ is the number of $q$'s appearing $(r_1, \cdots, r_{m-2})$. Since $(r_1, \cdots, r_{m-2})\in \mathfrak{D}_{m-1}^\sigma$ and $G_{m-1}$ is a number, $G_{m-1}$ must be a product of $j$ scalar products having either $\langle f_{n, i}, f_{n,j}\rangle$ or $(\cdots ((f_{n, j_1}\star_{(q+1)/2}^{(q-1)/2}f_{n, j_2})\star_{(q+1)/2}^{(q-1)/2})\cdots\star_{(q+1)/2}^{(q-1)/2}f_{n, j_{k}})\stackrel{q}\smallfrown f_{n, j_{k+1}}$. Now we show that
$$\lim_{n\rightarrow \infty}(\cdots ((f_{n, j_1}\star_{(q+1)/2}^{(q-1)/2}f_{n, j_2})\star_{(q+1)/2}^{(q-1)/2})\cdots\star_{(q+1)/2}^{(q-1)/2}f_{n, j_{k}})\stackrel{q}\smallfrown f_{n, j_{k+1}}=\lambda.\eqno (4.5)$$
By the hypotheses of this lemma and the proof of (2.4), we have
\begin{align*}
&\lim_{n\rightarrow \infty}|(\cdots((f_{n, j_1}\star_{(q+1)/2}^{(q-1)/2}f_{n, j_2})\star_{(q+1)/2}^{(q-1)/2} f_{n, j_3})\cdots\star_{(q+1)/2}^{(q-1)/2}f_{n, j_{k}})\stackrel{q}\smallfrown f_{n, j_{k+1}}-\lambda|\\
\le & \lim_{n\rightarrow\infty}|f_{n,j_{k}}\stackrel{q}\smallfrown f_{n, j_{k+1}}-\lambda|=0,
\end{align*}
We have proved (4.5), which implies $(4.4)$. Then following the proof of Lemma 4.4 in \cite{SB2} (and Lemma 5.2 in \cite{NP}), we  get
$$\sum_{i=0}^{\lfloor (m-2)/2\rfloor}\sum_{\sigma\in \mathfrak{S}_{2i+\pi(mq)}^{m-1}}\sum_{(r_1, \cdots, r_{m-2})\in \mathfrak{D}_{m-1}^\sigma}(\bigstar_{r_p}^{r_p-\sigma(p)}f_{n,p})_{p=1}^{m-2}\stackrel{q}\smallfrown f_{n,m}\rightarrow \sum_{j=1}^m\lambda^j R_{m,j},$$ as $n\rightarrow \infty$, where  $R_{m,j}$ is the number of non-crossing partitions in $NC_{\ge 2}(m)$ with exactly $j$ blocks.
On the other hand, $$\varphi(a_{i_1}\cdots a_{i_m})=\sum_{\pi\in NC_{\ge 2}(m)}\kappa_\pi(a_{i_1},\cdots, a_{i_m})=\sum_{\pi\in NC_{\ge 2}(m)}\lambda^{|\pi|}= \sum_{j=1}^m\lambda^j R_{m,j}.$$
\end{proof}
\begin{Theorem}
Let $q\ge 2$ be an integer, $\lambda>0$ and $\alpha:=\{\alpha_1, \alpha_2, \cdots\}$ be a sequence of non-zero real numbers,  and $\{a_{i}:i=1,2,\cdots\}$ be a  sequence of random variables in a non-commutative probability space $(\A, \varphi)$ with distribution $Z(\lambda, \alpha)$. Let $\{I(f_n^{(i)}):i, n\in \mathbb{N}\}$ be a bi-indexed sequence of multiple integrals of order $q$ with respect to the centred free Poisson random measure $\widehat{N}$, where $f_n^{(i)}$ is a symmetric function in $L^2(\mathbb{R}_+^q)$ such that there are numbers $M_i$ and $L_i$, and a $q$-dimensional square  $S_{n,i}$ such that $$\sup\{|f_{n}^{(i)}(x)|:x\in \mathbb{R}_+^q, n\in \mathbb{N}\}\le M_i, $$    $$\{x\in \mathbb{R}_+^{q}:f_{n}^{(i)}(x)\ne 0\} \subset S_{n,i}, \sup\{\mu (S_{n,i}):n=1, 2, \cdots\}\le L_i, $$ for $i\in \mathbb{N}$. Then the following two statements are equivalent.
\begin{enumerate}
\item $\{I(f_n^{(i)}): i=1, 2, \cdots\}$ converges in distribution to $\{a_i:i=1, 2, \cdots\}$ , as $n\rightarrow \infty$;
\item The equations $(4.1)$ and $(4.2)$ hold,  for all $i, j\in \mathbb{N}$, where the free integrals in  $(4.2)$ are defined with respect to the centered free Poisson random measure.
\end{enumerate}
\end{Theorem}
\begin{proof} As in the previous theorems, we only need to prove $(2)\Rightarrow (1)$ in the case that $\alpha_1=\alpha_2=\cdots =1$.   For $1<m\in \mathbb{N}$ and $\chi:\{1,2,\cdots\}\rightarrow \mathbb{N}$. It is sufficient to prove $$\lim_{n\rightarrow \infty}\varphi(I(f_n^{(\chi(1))})\cdots I(f_n^{(\chi(m))}))=\varphi(a_{\chi(1)}a_{\chi(2)}\cdots a_{\chi(m)}).\eqno (4.6)$$

Case I. $q$ is even. Equation $(4.2)$ implies that $$\lim_{n\rightarrow \infty}\varphi(I(f_n^{(i)})T(f_n^{(j)})^2I(f_n^{(i)})-2I(f_n^{(i)})T(f_n^{(j)})^2)=2\lambda^2-\lambda.\eqno(4.7)$$ It follows from (4.7), the proof of Lemma 4.3 in \cite{SB2} that $$\lim_{n\rightarrow \infty}(\|f_n^{(i)}\stackrel{q/2}\smallfrown f_n^{(j)}-f_n^{(j)}\|+\|f_n^{(i)}\stackrel{r}\smallfrown f_n^{(j)}\|+\|f_n^{(i)}\star_{s}^{s-1}f_n^{(j)}\|)=0,$$ for $1\le r<q$, $r\ne q/2$, and $1\le s\le q$. Then, exactly following the proof of $(2)\Rightarrow (1)$ in Theorem 3.3, we get (4.6).

Case II. $q$ is odd. By $(3.2)$, we have
\begin{align*}
&\varphi(I(f_{n}^{(\chi(1))})I(f_{n}^{(\chi(2))})\cdots I(f_{n}^{(\chi(m))}))\\
=&\sum_{i=0}^{\lfloor(m-2)/2\rfloor}\sum_{\sigma\in \mathfrak{S}_{2i+\pi(qm)}^{m-1}}\sum_{(r_1, \cdots, r_{m-2})\in \mathfrak{D}_{m-1}}^\sigma((\bigstar_{r_p}^{r_p-\sigma(p)}f_{n}^{(\chi(p))})_{p=1}^{m-2}\stackrel{q}\smallfrown f_{n}^{(\chi(m))})\\
+&\sum_{i=0}^{\lfloor(m-2)/2\rfloor}\sum_{\sigma\in \mathfrak{S}_{2i+\pi(qm)}^{m-1}}\sum_{(r_1, \cdots, r_{m-2})\in \mathfrak{E}_{m-1}^\sigma}((\bigstar_{r_p}^{r_p-\sigma(p)}f_{n}^{(\chi(p))})_{p=1}^{m-2}\stackrel{q}\smallfrown f_{n}^{(\chi(m))}).
\end{align*}
 Condition $(2)$, Lemmas 3.2 and 4.3 implies that the first sum in the right-hand side of the above equation approaches $\varphi(a_{\chi(1)}\cdots a_{\chi(m)})$, as $n\rightarrow \infty$. It remains to show that  $$(\bigstar_{r_p}^{r_p-\sigma(p)}f_{n}^{(\chi(p))})_{p=1}^{m-2}\stackrel{q}\smallfrown f_{n}^{(\chi(m))}\rightarrow 0, \eqno(4.8)$$ as $n\rightarrow \infty$, for all $(r_1, \cdots, r_{m-2})\in \mathfrak{E}_{m-1}^\sigma$.
Let $r_l$ be the first item in the sequence $r_1, r_2, \cdots, r_{m-2}$ which breaks the rules in $\mathfrak{D}_{m-2}^\sigma$, that is, for $j<l$, $r_j\in \{0, \frac{q+1}{2}, q\}$ and $r_j\in \{0, q\}$ if and only if $\sigma(j)=0$. We shall prove $(4.8)$ in the all possible cases as follows.

$(i)$. The number $r_l$ satisfies $1\le r_l<q$ and $r_l\ne (q+1)/2$. In this case, the equation $(4.8)$ follows from  the equation $\lim_{n\rightarrow\infty}\|f_n^{(i)}\star_{r}^{r-\sigma}f_n^{(j)}\|=0$, where $\sigma=0$ or $1$, $1\le r<q,  r\ne (q+1)/2$ (see Lemma 3.2 (2)), $(3.6)$, and the proof of Theorem 3.3.

$(ii)$. The number $r_l=q$ and $\sigma(l)=1$. In this case, the equation $(4.8)$ follows from  the equation $\lim_{n\rightarrow\infty}\|f_n^{(i)}\star_{q}^{q-1}f_n^{(j)}\|=0$ (see Lemma 3.2 (2)), $(3.6)$, and the proof of Theorem 3.3.

$(iii)$. The number $r_l=(q+1)/2$ and $\sigma(l)=0$. In this case, we have  $$\lim_{n\rightarrow \infty}\|f_{n, i_l}\stackrel{(q+1)/2}\smallfrown f_{n, i_{l+1}}\|=0,\eqno (4.9)$$ by Lemma 3.2 (2), since $ (q+1)/2\ne q$. The equation  $(4.8)$ follows by (4.9), (3.6) and the proof of Theorem 3.3.
\end{proof}

\end{document}